\documentclass[12pt,a4paper]{article}
\usepackage{a4,a4wide}
\usepackage{amsfonts,amsmath,amssymb,amsthm}
\usepackage{color}
\usepackage{ifpdf}
\ifpdf
    \usepackage[pdftex]{graphicx}
    \DeclareGraphicsExtensions{.png,.pdf,.jpg}
\else
   \usepackage[dvips]{graphicx}
   \DeclareGraphicsExtensions{.eps}
\fi
\graphicspath{{.}{figuresPL/}}
\usepackage[latin1]{inputenc}
\numberwithin{equation}{section}

\newtheorem{theorem}{Theorem}[section]
\newtheorem{lemma}[theorem]{Lemma}
\newtheorem{proposition}[theorem]{Proposition}

\newtheorem{remark}[theorem]{Remark}
\renewcommand\tilde{\widetilde}

\def\R{\mathbb{R}}

\def\E{\mathcal{E}}

\def\H{\mathcal{H}}

\renewcommand{\phi}{\varphi}

\def\Div{\textup{div}\,}

\def\1{\mathbf{1}}

\def\Xint#1{\mathchoice
{\XXint\displaystyle\textstyle{#1}}%
{\XXint\textstyle\scriptstyle{#1}}%
{\XXint\scriptstyle\scriptscriptstyle{#1}}%
{\XXint\scriptscriptstyle\scriptscriptstyle{#1}}%
\!\int}
\def\XXint#1#2#3{{\setbox0=\hbox{$#1{#2#3}{\int}$ }
\vcenter{\hbox{$#2#3$ }}\kern-.57\wd0}}

\def\dashint{\Xint-}

\def\eps{\varepsilon}
\renewcommand{\subset}{\subseteq}

\def\lt{\left}
\def\rt{\right}
\def\les{\lesssim}
\def\ges{\gtrsim}

\def\E{\mathcal{E}}

\def\Bf{\overline{f}}
\def\supp{\textup{spt}\,}
\def\delrho{\delta\!\rho}

\begin{document}
\title{A variational proof of partial regularity for optimal transportation maps}
\author{M. Goldman \and F. Otto} 
\maketitle
\begin{abstract}\noindent
We provide a new proof of the known partial regularity result for the optimal transportation map (Brenier map) between two H\"older continuous densities.
Contrary to the existing regularity theory for the Monge-Amp\`ere equation, which is based on the maximum principle, our approach is purely variational. 
By constructing a competitor on the level
of the Eulerian (Benamou-Brenier) formulation, we show that locally, the velocity is close to the gradient of a harmonic function 
provided the transportation cost is small. We then translate back to the Lagrangian description and perform a Campanato
iteration to obtain an $\eps$-regularity result.
%Combining both Eulerian and Lagrangian points of view, we set up a Campanato iteration scheme to obtain an $\eps-$regularity result. 
%The heart of our proof lies in proving that at each step,  the  transport map is closer and closer to be the gradient of a harmonic function. 
\end{abstract}

\section{Introduction}

For $\alpha\in(0,1)$, let $\rho_0$ and $\rho_1$ be two  probability densities with bounded support which are  $C^{0,\alpha}$ continuous, bounded and bounded away from zero on their support
and let $T$ be the solution of the optimal transportation problem
\begin{equation}\label{eq:introprobOT}
 \min_{T\sharp \rho_0=\rho_1} \int_{\R^d} |T(x)-x|^2 \rho_0(x)dx,
\end{equation}
where with a slight abuse of notation $T\sharp \rho_0$ denotes the push-forward by $T$ of the measure $\rho_0 dx$ 
(existence and characterization of $T$ as the gradient of a convex function $\psi$
are given by Brenier's Theorem, see \cite[Th. 2.12]{Viltop}). Our main result is a partial regularity theorem for $T$:
\begin{theorem}\label{theomainintro}
There exist open sets $E\subset \supp \rho_0$ and $F\subset \supp \rho_1$ of full measure such that $T$ is a $C^{1,\alpha}$-diffeomorphism between $E$ and $F$.
\end{theorem}
This theorem is a consequence of Alexandrov Theorem \cite[Th. 14.25]{VilOandN} and the following $\eps-$regularity theorem 
(plus a  bootstrap argument):
\begin{theorem}\label{theoepsintro}
 Let $T$ be the minimizer of \eqref{eq:introprobOT} and assume that $\rho_0(0)=\rho_1(0)=1$. There exists $\eps(\alpha,d)$
 such that if\footnote{here  $[\rho]_{\alpha,R}:=\sup_{x,y\in B_R} \frac{|\rho(x)-\rho(y)|}{|x-y|^\alpha}$ denotes the $C^{0,\alpha}-$semi-norm.}
 \[
  \frac{1}{(2R)^{d+2}} \int_{B_{2R}} |T-x|^2 \rho_0dx+\frac{1}{(2R)^{d+2}} \int_{B_{2R}} |T^{-1}-x|^2 \rho_1dx +R^{2\alpha}[\rho_0]_{\alpha,2R}^2+R^{2\alpha}[\rho_1]^2_{\alpha,2R}\le \eps(\alpha),
 \]
then, $T$ is $C^{1,\alpha}$ inside $B_R$.
\end{theorem}

Theorem \ref{theomainintro} was already obtained by Figalli and Kim \cite{FigKim} (see also \cite{DePFig} for a far-reaching generalization),
but our proof departs from the usual scheme for proving regularity for the Monge-Amp\`ere equation. 
Indeed, while most proofs use some variants of the maximum principle, our proof is variational. The classical approach operates on the level of the 
convex potential $\psi$ and the ground-breaking paper in that respect is Caffarelli's \cite{CafAoM90bis}: By comparison with simple barriers it is shown that an
Alexandrov (and thus viscosity) solution $\psi$ to the Monge-Amp\`+ere equation is $C^1$,
provided its convexity does not degenerate along a line crossing the entire domain of definition. 
The same author shows in \cite{CafJAMS92} by similar arguments that the potential $\psi$ of the Brenier map is a strictly convex Alexandrov solution,
and thus regular, provided the target domain $\supp\rho_1$ is convex. The challenge of the $\eps-$regularity theorem in \cite{FigKim} is to follow
the above line of arguments while avoiding the notion of Alexandrov solution, that is, without having access to the comparison argument by below.
The $\eps-$regularity theorem in \cite{FigKim} in turn is used by Figalli and De Philippis as the core for a generalization to general cost functions
by means of a Campanato iteration.
%The classical approach operates on the level of the convex potential $\psi$, and the main difficulty is to prove that a Brenier solution
%to the Monge-Amp\`ere equation is actually a strictly convex Alexandrov solution, at which point Caffarelli's regularity theory 
%\cite{CafCPAM91} applies (see in particular \cite{FigKim}, or \cite{DePFig} where an $\eps-$regularity theorem based on a Campanato iteration is proven).
%Maximum principle arguments are also underlying the global regularity results of Caffarelli \cite{CafAoM90,CafJAMS92,CafAoM90bis}. 
On the contrary to these papers, we work directly at the level of the optimal transportation map $T$, and besides the $L^\infty$ bound \eqref{eq:displconv} 
given by McCann's displacement convexity, we only use variational arguments. The main idea behind the proof is the well-known fact that
the linearization of the Monge-Amp\`ere equation gives rise to the Laplace equation \cite[Sec. 7.6]{Viltop}. 
We prove that if the energy in a given ball is small enough, then in the half-sized ball, $T$ is close to the gradient of harmonic function 
(see Proposition \ref{prop:Lagestim}). 
This result is actually established at the Eulerian level
(i.e. for the solutions of the Benamou-Brenier formulation of optimal transportation, see \cite[Th. 8.1]{Viltop}  or \cite[Chap. 8]{AGS}),
see Proposition \ref{Prop:BBharm}. It is for this result that we need the outcome of McCann's displacement convexity, cf.  \eqref{eq:displconv}, 
since it is required for the quasi-orthogonality property \eqref{secondreduc}. Our argument is variational and proceeds by 
defining a competitor based on the solution of a Poisson equation with suitable flux boundary conditions, 
and a boundary-layer construction. The boundary-layer construction is carried out in Lemma \ref{lemLambda};
by a duality argument it reduces to the trace estimate \eqref{trace}. This part of the proof is reminiscent of arguments from \cite{ACO}.
Once we have this approximation result, using that by classical elliptic regularity,  harmonic functions are close to their second-order Taylor expansion, 
we establish ``improvement of flatness by tilting'', see Proposition \ref{iter}. This means that if the energy in a given ball is small
then, up to a change of coordinates, the energy has a geometric decay on a smaller scale.    
The last step is to perform a Campanato iteration of this one-step improvement. This is done in Proposition \ref{P3}, 
where we use our last fundamental ingredient, namely the invariance of the variational problem under affine transformations.
This entire approach to $\eps$-regularity is guided by De Giorgi's strategy for minimal surfaces (see \cite{Maggi} for instance).\\
Let us notice that because of the natural scaling of the problem, our Campanato iteration operates directly at the $C^{1,\alpha}$-level for $T$,
as opposed to \cite{FigKim,DePFig}, where $C^{0,\alpha}$-regularity is obtained first. \\

The plan of the paper is the following. In Section \ref{sec:not} we gather some notation that we will use throughout the paper. 
Then, in Section \ref{sec:prelim}, we recall some well-known facts
about the Poisson equation and then prove estimate \eqref{eq:estimminLambda}, the proof of which is based on the trace estimate \eqref{trace}. 
In the final section, we prove Theorem \ref{theoepsintro} and then Theorem \ref{theomainintro}.\\
% Before closing this introduction,
% let us point out that the global scheme of our proof shares some features with the proof of \cite{DePFig}. The main difference is that in their $\eps-$regularity
% result, they compare their solutions to strictly convex Alexandrov solutions of the classical Monge-Amp\`ere equation $\det \nabla^2 v=1$, while we compare our solutions to harmonic functions.      

Motivated by applications to the optimal matching problem, we are currently working together with M. Huesmann on the extension of Proposition \ref{prop:Lagestim}  to arbitrary target measures. 
A previous version of this paper treating the simpler case of transportation between sets is available on our webpages.
Since  the proofs are more streamlined there, we recommend to read it first.   

\section{Notation}\label{sec:not}
In the paper we will use the following notation. The symbols $\sim$, $\ges$, $\les$ indicate estimates that hold up to a global constant $C$,
which typically only depends on the dimension $d$ and the H\"older exponent $\alpha$ (if applicable). 
For instance, $f\les g$ means that there exists such a constant with $f\le Cg$,
$f\sim g$ means $f\les g$ and $g\les f$. An assumption of the form $f\ll1$ means that there exists $\eps>0$, typically only
depending on dimension and the H\"older exponent, such that if $f\le\eps$, 
then the conclusion holds.  
%We denote by $\H^k$ the $k-$dimensional Hausdorff measure. 
%For a set $E$,  $\nu_E$ will always denote the external normal to $E$. When clear from the context we will drop the explicit dependence on the set. 
We write $|E|$ for the  Lebesgue measure of a set $E$. Inclusions will always be understood as holding up to a set of Lebesgue measure zero, 
that is for two sets $E$ and $F$, $E\subset F$ means that $|E\backslash F|=0$.   
When no confusion is possible, we will drop the integration measures in the integrals. 
For $R>0$ and $x_0\in \R^d$, $B_R(x_0)$ denotes the ball of radius $R$ centered in $x_0$. 
When $x_0=0$, we will simply write $B_R$ for $B_R(0)$. We will also use the notation
\[\dashint_{B_R} f:=\frac{1}{|B_R|}\int_{B_R} f.\]
For a function $\rho$ defined on a ball $B_R$ we introduce the H\"older semi-norm
of exponent $\alpha\in(0,1)$
\[
 [\rho]_{\alpha,R}:=\sup_{x\not=y\in B_R}\frac{|\rho(x)-\rho(y)|}{|x-y|^\alpha}.
\]
%We then define  the H\"older norm
%\[\|\rho\|_{C^{0,\alpha}(B_R(x_0))}:= \|\rho\|_{L^\infty(B_R(x_0))}+[\rho]_{C^{0,\alpha}(B_R(x_0))}.\]
%Similarly, we define 
%\[
% \|\rho\|_{C^{1,\alpha}(B_R(x_0))}:=\|\nabla \rho\|_{C^{0,\alpha} (B_R(x_0))}+\|\rho\|_{L^\infty(B_R(x_0)}.
%\]
%Higher H\"older norms are then defined in a natural way.

\section{Preliminaries}\label{sec:prelim}
In this section, we first recall some well-known estimates for harmonic functions. 
\begin{lemma}\label{lem:harm}
Given $f\in L^{2}(\partial B_1)$ with average zero, we consider a solution $\phi$ of 
\begin{equation}\label{poisson}
 \begin{cases}
  -\Delta \phi=0 & \textrm{in } B_1\\
  \frac{\partial \phi}{\partial \nu}= f & \textrm{on } \partial B_1, 
 \end{cases} 
\end{equation}
where $\nu$ denotes the outer normal to $\partial B_1$. We have 
\begin{equation}\label{energieestimphi}
 \int_{B_1} |\nabla \phi|^2\les \int_{\partial B_1} f^2,
\end{equation}

\begin{equation}\label{Schauder}
\sup_{B_{1/2}}\big(|\nabla^3\phi|^2+|\nabla^2\phi|^2+|\nabla\phi|^2\big)  \les \int_{ B_1} |\nabla \phi|^2,
\end{equation}
and for every $r\le 1$, letting $A_r:=B_1\backslash B_{1-r}$,
\begin{equation}\label{estimaPoisannulus}
 \int_{A_r} |\nabla \phi|^2\les r\int_{\partial B_1} f^2.
\end{equation}

\end{lemma}
\begin{proof}
We start with \eqref{energieestimphi}. Changing $\phi$ by an additive constant, we may assume that $\int_{ B_1} \phi=0$. Testing
(\ref{poisson}) with $\phi$, we obtain
 \begin{align*}
  \int_{B_1}|\nabla \phi|^2&=\int_{\partial B_1} f \phi \\
  &\le \lt(\int_{\partial B_1} f^2\rt)^{1/2} \lt(\int_{\partial B_1} \phi^2\rt)^{1/2}\\
  &\les  \lt(\int_{\partial B_1} f^2\rt)^{1/2} \lt(\int_{ B_1} |\nabla \phi|^2\rt)^{1/2},
 \end{align*}
where we used the trace estimate in conjunction with Poincar\'e's estimate for mean-value zero. This yields \eqref{energieestimphi}.\\
Estimate \eqref{Schauder} follows from the mean-value property of harmonic functions applied to $\nabla\phi$ and its derivatives.\\
We finally turn to \eqref{estimaPoisannulus}. 
%By harmonicity of $\nabla \phi$,
%\begin{align*}
% \frac{d}{dr} \frac{1}{r^{d-1}}\int_{\partial B_r} |\nabla \phi|^2&=\frac{d}{dr} \int_{\partial B_1} |\nabla \phi|^2(r \sigma)\\
% &=2\int_{\partial B_1} (\nabla \phi\cdot(D^2\phi \nu)) (r\sigma)\\
% &=\frac{2}{r^{d-1}} \int_{\partial B_r} \scp{ D^2\phi\nabla \phi}{ \nu}\\
% &=\frac{2}{r^{d-1}} \int_{B_r} \Div \lt(D^2\phi\nabla \phi\rt)=\frac{2}{r^{d-1}} \int_{B_r} |D^2 \phi|^2\ge 0,
%\end{align*}
By sub-harmonicity of $|\nabla\phi|^2$ (which can for instance be inferred from the Bochner formula), we have the mean-value property in the form
\[
 \int_{\partial B_r} |\nabla \phi|^2\le \int_{\partial B_1} |\nabla \phi|^2\quad\mbox{for}\;r\le 1.
\]
Integrating this inequality between $r$ and $1$, using Pohozaev identity, that is,
\begin{equation}\label{pohozaev}
(d-2)\int_{B_1} |\nabla \phi|^2= \int_{\partial B_1} \lt|\nabla_\tau \phi\rt|^2
- \int_{\partial B_1} \lt(\frac{\partial \phi}{\partial \nu}\rt)^2,
\end{equation}
where $\nabla_\tau$ is the tangential part of the gradient of $\phi$, and \eqref{energieestimphi}, we obtain \eqref{estimaPoisannulus}.
\end{proof}
We also need similar estimates for solutions of Poisson equation.
\begin{lemma}\label{lem:Poi}
Given $g\in C^{0,\alpha}(\overline{B}_1)$ such that $g(0)=0$,
we consider a solution $\phi$ of 
\begin{equation}\label{Poi:poisson}
 \begin{cases}
  -\Delta \phi=g & \textrm{in } B_1\\
  \frac{\partial \phi}{\partial \nu}= -\frac{1}{\H^{d-1}(\partial B_1)}\int_{B_1} g & \textrm{on } \partial B_1, 
 \end{cases} 
\end{equation}
where $\nu$ denotes the outer normal to $\partial B_1$. We have 
\begin{equation}\label{Poi:Schauder}
\sup_{B_{1}}\big(|\nabla^2\phi|^2+|\nabla\phi|^2\big)  \les [g]_{\alpha,1}^2.
\end{equation}
In particular,
\begin{equation}\label{Poi:energieestimphi}
 \int_{B_1} |\nabla \phi|^2\les [g]_{\alpha,1}^2,
\end{equation}
and  letting for $r\le1$, $A_r:=B_1\backslash B_{1-r}$, it holds
\begin{equation}\label{Poi:estimaPoisannulus}
 \int_{A_r} |\nabla \phi|^2\les r[g]_{\alpha,1}^2.
\end{equation}

\end{lemma}
\begin{proof}
 Estimate \eqref{Poi:Schauder} follows from global Schauder estimates \cite{Nardi} and the fact that since $g(0)=0$, $\|g\|_{L^\infty(B_1)}\les [g]_{\alpha,1}$.
\end{proof}

We will  need a trace estimate in the spirit of \cite[Lem. 3.2]{ACO}.
\begin{lemma}
For $r\le 1$, letting $A_r:=B_1\backslash B_{1-r}$, it holds for every function $\psi$,
 \begin{equation}\label{trace}
 \lt(\int_0^1\int_{\partial B_1}(\psi-\overline{\psi})^2\rt)^{1/2}\les r^{1/2} \lt(\int_{0}^1\int_{A_r} |\nabla \psi|^2\rt)^{1/2}+\frac{1}{r^{(d+1)/2}} \int_0^1\int_{A_r} |\partial_t \psi|,
\end{equation}
where $\overline{\psi}(x):=\int_0^1\psi(t,x) dt$. 
\end{lemma}
\begin{proof}
By a standard density argument, we may assume $\psi\in C^1(\overline{A_r}\times[0,1])$. \\
Because of $\int_{0}^1|\nabla(\psi-\overline{\psi})|^2$
$\le\int_{0}^1|\nabla\psi|^2$, we may rewrite \eqref{trace} in terms of $v:=\psi-\overline{\psi}$ as
\[
  \lt(\int_0^1\int_{\partial B_1}v^2\rt)^{1/2}\les r^{1/2} \lt(\int_{0}^1\int_{A_r} |\nabla v|^2\rt)^{1/2}+\frac{1}{r^{(d+1)/2}} \int_0^1\int_{A_r} |\partial_t v|.
\]
Since for every $x\in \partial B_1$, $\int_0^1 v=0$, we have 
$\lt(\int_0^1 v^2\rt)^{1/2}$ $\le\int_0^1 |\partial_t v|$,
 so that it is enough to prove 
 \[
  \lt(\int_{\partial B_1} \int_0^1v^2\rt)^{1/2}\les r^{1/2} \lt(\int_{A_r}\int_{0}^1 |\nabla v|^2\rt)^{1/2}+\frac{1}{r^{(d+1)/2}} \int_{A_r} \lt(\int_0^1 v^2\rt)^{1/2}.
 \]
Introducing $V:=\lt(\int_0^1 v^2\rt)^{1/2}$ and noting that $|\nabla V|^2\le \int_0^1|\nabla v|^2$, we see that it is sufficient to establish
\begin{equation}\label{tracereduce}
 \lt(\int_{\partial B_1} V^2\rt)^{1/2}\les r^{1/2} \lt(\int_{A_r} |\nabla V|^2\rt)^{1/2} +\frac{1}{r^{(d+1)/2}} \int_{A_r} |V|.
\end{equation}
We now cover the sphere $\partial B_1$ by (geodesic) cubes $Q$ of side-length $\sim r$ in such a way that
there is only a locally finite overlap. Then the annulus $A_r$ is covered by the corresponding conical sets $Q_r$. By summation over $Q$
and the super-additivity of the square function, for (\ref{tracereduce}) it is enough to prove for every $Q$
\begin{equation}\nonumber
 \lt(\int_{Q} V^2\rt)^{1/2}\les r^{1/2} \lt(\int_{Q_r} |\nabla V|^2\rt)^{1/2} +\frac{1}{r^{(d+1)/2}} \int_{Q_r} |V|.
\end{equation}
Since $Q_r$ is the bi-Lipschitz image of the Euclidean cube $(0,r)^d$, it is enough to establish
%Let $N:=\lfloor \frac{2\pi}{r}\rfloor$ and let $\phi_1,\dots, \phi_{d-1}\in[0,2\pi]$ be the $(d-1)$-dimensional spherical coordinates. 
%We consider a partition of the annulus $A_r$ into sets of the form
%\[Q_i:= \lt\{ x \in A_r  \ : \ \forall j\in [1,d-1], \ \phi_j\in \lt[ \frac{2\pi}{N} n_{i,j} ,  \frac{2\pi}{N} (n_{i,j}+1)\rt]\rt\},\]
%where for $1\le i\le N^{d-1}$ and $1\le j\le d-1$, $n_{i,j}$ goes from $0$ to $N-1$. We claim that for every $i$,
%\[
% \int_{\partial Q_i} V^2\les r\int_{Q_i}|\nabla V|^2+\frac{1}{r^{d+1}}\lt(\int_{Q_i} V\rt)^2,
%\]
%which, by summing over $i$ and using the super-additivity of the square would yield \eqref{tracereduce}.
%Since, for every $i$, there exists a diffeomorphism $\Psi_i:Q_i\to [0,r]^d$ with 
%$\|\nabla \Psi_i\|_\infty+ \|\det \nabla \Psi_i\|_\infty \sim 1$, it is enough to show that
\begin{equation}\label{tracefinal}\int_{\{0\}\times(0,r)^{d-1}} V^2\les r\int_{(0,r)^d}|\nabla V|^2+\frac{1}{r^{d+1}}\lt(\int_{(0,r)^d}|V|\rt)^2.
\end{equation}
By rescaling, for (\ref{tracefinal}) it is sufficient to consider $r=1$. 
%Consider one of the faces of $\partial [0,1]^d$, say $\{0\}\times[0,1]^{d-1}$. Then, 
By a one-dimensional trace inequality we have for every $x'\in (0,1)^{d-1}$
\[ |V(0,x')|\les \int_{0}^1 |\partial_1 V(x_1,x')|dx_1 +\int_0^1 |V(x_1,x')| dx_1.\]
Taking squares, integrating and using Jensen's inequality, we get
\[\int_{\{0\}\times(0,1)^{d-1}} V^2\les \int_{(0,1)^d} |\partial_1 V|^2+ \int_{(0,1)^{d}} V^2.\]
Using   Poincar\'e inequality in the form $\int_{(0,1)^d} V^2\les \int_{(0,1)^d} |\nabla V|^2+ \lt(\int_{(0,1)^d} |V|\rt)^2$, we
obtain (\ref{tracefinal}). 
%\[\int_{\{0\}\times(0,1)^{d-1}} V^2\les \int_{[0,1]^d} |\nabla  V|^2+ \lt(\int_{[0,1]^d} |V|\rt)^2.\]
%Summing this inequality over each face of $\partial [0,1]^d$, we conclude the proof of \eqref{tracefinal}, which in turn closes the proof of \eqref{trace}.
\end{proof}

This trace estimate is used in a similar spirit as in \cite[Lem. 3.3]{ACO} to obtain
\begin{lemma}\label{lemLambda}
 Let $f\in L^2( \partial B_1\times(0,1))$ be such that for a.e. $x\in \partial B_1$, $\int_0^1 f(x,t)dt=0$. For $r>0$ we introduce 
 $A_r:=B_1\backslash B_{1-r}$ and define $\Lambda$ as 
 the set of pairs $(s,q)$ with $|s|\le 1/2$ and such that for $\psi\in C^1(\overline{B_1}\times[0,1])$\footnote{For  $(s, q) $ regular, \eqref{localcontieq} just means
 $\partial_t s +\Div q=0$   in $A_r$,    $s(\cdot,0)= s(\cdot,1)=0$,
  $q\cdot \nu =0$  on $\partial {B_{1-r}}\times(0,1)$  and   
 $q\cdot \nu  = f$  on  $\partial B_1\times(0,1)$},
\begin{equation}\label{localcontieq} \int_0^1\int_{A_r} s \partial_t\psi+q\cdot \nabla \psi=
\int_0^1\int_{\partial B_1} f \psi.    \end{equation}
Provided  $r\gg \lt(\int_0^1\int_{\partial B_1} f^2\rt)^{1/(d+1)}$ we then have
 \begin{equation}\label{eq:estimminLambda}
  \inf_{(s,q)\in\Lambda} \int_0^1\int_{A_r} \frac{1}{2}| q|^2\les r\int_0^1\int_{\partial B_1} f^2.
 \end{equation}
\end{lemma}

\begin{proof}
We first note that the class $\Lambda$ is not empty: For $t\in(0,1)$, let $u_t$ be defined as the (mean-free) solution of
the Neumann problem
 \[ \begin{cases}-\Delta u_t = -\frac{1}{|A_r|} \int_{\partial B_1} f & \textrm{in } A_r\times(0,1) \\[8pt]
  \frac{ \partial u_t}{\partial \nu}=f &\textrm{on } \partial B_1\times(0,1) \\
  \frac{\partial u_t}{\partial \nu}=0 & \textrm{on } \partial B_{1-r}\times(0,1),\end{cases}\]
and set $ q(x,t):=\nabla u_t(x)$. The definition $ s(x,t)$ $:=-\int_0^t \Div  q(x,z) dz$ $=-\frac{1}{|A_r|}\int_{0}^t \int_{\partial B_1} f$
then ensures that (\ref{localcontieq}) is satisfied, and $r\gg \lt(\int_0^1\int_{\partial B_1} f^2\rt)^{\frac{1}{2}}$ yields $|s|\le1/2$.

As in \cite[Lem. 3.3]{ACO}, we now prove \eqref{eq:estimminLambda} with help of duality:
\begin{align*}
 \inf_{(s,q)\in\Lambda} \int_0^1\int_{A_r} \frac{1}{2 }| q|^2&=\inf_{(s,q), |s|\le1/2} \sup_\psi \lt\{ \int_0^1\int_{A_r} \frac{1}{2}| q|^2-\int_0^1\int_{A_r} s\partial_t \psi +q\cdot\nabla \psi\rt.\\
 &\qquad \qquad \lt.+\int_0^1\int_{\partial B_1} f\psi\rt\}\\
 &= \sup_\psi  \inf_{(s,q), |s|\le1/2}\lt\{ \int_0^1\int_{A_r} \frac{1}{2 }| q|^2-\int_0^1\int_{A_r} s\partial_t \psi +q\cdot\nabla \psi\rt.\\
 &\lt.\qquad \qquad +\int_0^1\int_{\partial B_1} f\psi\rt\},
\end{align*}
where the swapping of the $\sup$ and $\inf$ is allowed since the functional is convex in $(s,q)$ and linear in $\psi$ (see for instance \cite[Prop. 1.1]{Brezis}). 
Minimizing in $(s,q)$, and using $\int_0^1 f=0$ which allows us to smuggle in $\overline{\psi}:=\int_{0}^1\psi$, we obtain
\begin{align*}
 \inf_{(s,q)\in\Lambda} \int_0^1\int_{A_r} \frac{1}{2}| q|^2&=\sup_\psi  \lt\{ -\int_0^1\int_{A_r} \frac{1}{2}(|\nabla \psi|^2+|\partial_t \psi|) + \int_0^1\int_{\partial B_1} f\psi \rt\}\\
 &=\sup_\psi  \lt\{ -\int_0^1\int_{A_r} \frac{1}{2}(|\nabla \psi|^2+|\partial_t \psi|) + \int_0^1\int_{\partial B_1} f(\psi-\bar \psi) \rt\}\\
&\le \sup_\psi \lt\{ -\int_0^1\int_{A_r} \frac{1}{2}(|\nabla \psi|^2+|\partial_t \psi|) \rt.\\
&\qquad  \qquad \lt. + \lt(\int_0^1\int_{\partial B_1} f^2\rt)^{1/2}\lt(\int_0^1\int_{\partial B_1}(\psi-\bar \psi)^2\rt)^{1/2} \rt\} .
\end{align*}

With the abbreviation $F:= \lt(\int_0^1\int_{\partial B_1} f^2\rt)^{1/2}$ we have just established the inequality
\begin{equation*}
  \inf_{(s,q)\in\Lambda} \int_0^1\int_{A_r} \frac{1}{2}| q|^2\le \sup_{\psi}\lt\{F\lt(\int_0^1\int_{\partial B_1}(\psi-\overline{\psi})^2\rt)^{1/2} -\frac{1}{2} \int_0^1\int_{A_r}|\nabla \psi|^2 +|\partial_t \psi|\rt\}.
\end{equation*}
Using now \eqref{trace}, where we denote the constant by $C_0$, and Young's inequality, we find that 
provided $r\ge \lt(2C_0 F\rt)^{2/(d+1)}$ (in line with our assumption $r\gg \lt(\int_0^1\int_{\partial B_1} f^2\rt)^{1/(d+1)})$,
\begin{align*}
  \inf_{(s,q)\in\Lambda} \int_0^1\int_{A_r} \frac{1}{2}| q|^2&\le \sup_{\psi} \lt\{ \frac{1}{2}C_0^2 F^2 r +C_0\frac{F}{r^{(d+1)/2}}\int_0^1\int_{A_r} |\partial_t \psi| - \frac{1}{2} \int_0^1\int_{A_r} |\partial_t \psi|\rt\}\\
  &\les F^2 r= r\int_0^1\int_{\partial B_1} f^2. 
\end{align*}
This concludes the proof of \eqref{eq:estimminLambda}.
\end{proof}

\section{Proofs of the main results}
Let $\rho_0$ and $\rho_1$ be two  densities with compact support in $\R^d$  and equal mass and let $T$ be the minimizer of 
\begin{equation}\label{prob:OT}
 \min_{T\sharp \rho_0=\rho_1} \int_{\R^d} |T(x)-x|^2 \rho_0(x)dx,
\end{equation}
where by a slight abuse of notation $T\sharp \rho_0$ denotes the push-forward by $T$ of the measure $\rho_0 dx$. If $T'$ is the optimal transportation map between $\rho_1$ and $\rho_0$, then (see for instance \cite[Rem. 6.2.11]{AGS})
\begin{equation}\label{eq:inverseT}
 T'(T(x))=x, \qquad \textrm{and} \qquad T(T'(y))=y \qquad \textrm{for a.e. } (x,y)\in \supp \rho_0\times \supp \rho_1.
\end{equation}
By another abuse of notation, we will denote $T^{-1}:=T'$. \\

Now for $t\in[0,1]$ and $x\in \R^d$ we set $T_t(x):=t T(x)+(1-t) x$ and consider the non-negative and $\mathbb{R}^d$-valued measures respectively defined through
\begin{equation}\label{eq:defjrho}
 \rho(\cdot,t):=T_t\sharp \rho_0 \qquad \textrm{and} \qquad j(\cdot,t):= T_t\sharp \lt[(T-Id)\rho_0\rt].
\end{equation}
%For every $t\in[0,1]$
%, both $\rho(\cdot,t)$ and $j(\cdot,t)$ are absolutely continuous with respect to the Lebesgue measure and furthermore, 
It is easy to check that $j(\cdot,t)$ is absolutely continuous with respect to $\rho(\cdot,t)$.
The couple $(\rho,j)$ solves the Eulerian (or Benamou-Brenier) formulation of optimal transportation (see \cite[Th. 8.1]{Viltop}  or \cite[Chap. 8]{AGS}, see also \cite[Prop. 5.32]{Santam} for the uniqueness), i.e. it is  the minimizer of 
\begin{equation}\label{BenamBre}
\min \lt\{ \int_0^1\int_{\R^d} \frac{1}{\rho}|j|^2 \ : \ \partial_t \rho+\Div j=0, \quad \rho(\cdot,0)=\rho_0, \ \rho(\cdot,1)=\rho_1\rt\}, 
\end{equation}
where the continuity equation including its boundary conditions are imposed in a distributional sense
and where the functional is defined through (see \cite[Th. 2.34]{AFP}),
\[
 \int_0^1\int_{\R^d} \frac{1}{\rho}|j|^2:=\begin{cases}
                                          \displaystyle \int_0^1\int_{\R^d} \lt|\frac{dj}{d\rho}\rt|^2 d\rho & \textrm{if } j\ll \rho,\\[10pt]
                                           +\infty & \textrm{otherwise}.
                                          \end{cases}
\]
%Let us point out that when  $\rho$ and $j$ are both absolutely continuous with respect to the Lebesgue measure with $j\ll \rho$, letting 
%\[
% \frac{1}{\rho}|j|^2 (x,t):=\begin{cases}
%                           \frac{1}{\rho(x,t)}|j(x,t)|^2 & \textrm{if } \rho(x,t)\neq 0\\
%                           0 & \textrm{otherwise},
%                         \end{cases}
%\]
%then 
%\[
% \int_0^1\int_{\R^d} \frac{1}{\rho}|j|^2=\int_0^1\int_{\R^d} \frac{1}{\rho}|j|^2(x,t) dx dt.
%\]
Since $T$ is the gradient of a convex function, by Alexandrov Theorem \cite[Th. 14.25]{VilOandN}, $T$ is differentiable a.e., that is  for a.e. $x_0$, there exists a symmetric matrix $A$ such that 
\[
  T(x)=T(x_0)+ A(x-x_0)+o(|x-x_0|).
\]
Moreover, $A$ coincide a.e. with the absolutely continuous part of the distributional derivative $D T$ of the map $T$. We will from now on denote $\nabla T(x_0):=A$.
For $t\in[0,1]$, by  \cite[Prop. 5.9]{Viltop}, $\rho(\cdot,t)$ (and thus also $j$) is absolutely continuous with respect to the Lebesgue measure. The functional can be therefore rewritten as
\[
 \int_0^1\int_{\R^d} \frac{1}{\rho}|j|^2=\int_0^1\int_{\R^d} \frac{1}{\rho}|j|^2(x,t) dx dt,
\]
where
\[
 \frac{1}{\rho}|j|^2 (x,t):=\begin{cases}
                           \frac{1}{\rho(x,t)}|j(x,t)|^2 & \textrm{if } \rho(x,t)\neq 0\\ 
                           0 & \textrm{otherwise}.
                          \end{cases}
\]
Moreover,   the Jacobian equation
\begin{equation}\label{eq:Jacob}
 \rho(t, T_t(x))\det \nabla T_t(x)=\rho_0(x),
\end{equation}
holds a.e. (see \cite[Ex. 11.2]{VilOandN} or \cite[Th. 4.8]{Viltop}) and in particular, $\rho_1(T(x))\det \nabla T(x)=\rho_0(x)$.

% By concavity of $\det(\cdot)^{1/d}$ on non-negative symmetric matrices, we get that
% \begin{equation}\label{displconv}
%  \rho\le 1,
% \end{equation}
% which is a special instance of McCann's displacement convexity (see \cite[Cor. 4.4]{McCann}). 
The proof of Theorem \ref{theoepsintro} is based on the decay properties of the excess energy
\begin{equation}\label{def:excess}
 \E(\rho_0,\rho_1,T,R):=R^{-2}\dashint_{B_R} |T-x|^2 \rho_0.
\end{equation}
As will be shown in the proofs of Proposition \ref{P3} and Theorem \ref{theo:final}, up to a change of variables it is not restrictive to assume that $\rho_0(0)=\rho_1(0)=1$.\\
Before proceeding with the proof of Theorem \ref{theoepsintro}, let us gather few useful lemmas.
\begin{lemma}\label{lem:Linfty}
 Let $T$ be a minimizer of \eqref{prob:OT} and assume that  $\frac{1}{2}\le \rho_0\le 2$ in $B_1$ and that  $\E(\rho_0,\rho_1,T,1)\ll1$. Then,
 \begin{equation}\label{supestim}\sup_{B_{3/4}} |T-x|+ |T^{-1}-x|\les \lt(\int_{B_1} |T-x|^2\rho_0 \rt)^{1/(d+2)}.\end{equation}
As a consequence,
 \begin{equation}\label{Ttinclu}
  T_t(B_{1/8})\subset B_{3/16}.
 \end{equation}
 Moreover, for $t\in[0,1]$, we have for the pre-image
 \begin{equation}\label{infestimTt}
  T_t^{-1}(B_{1/2})\subset B_{3/4}.
 \end{equation}
\end{lemma}
\begin{proof}
We begin with the proof of \eqref{supestim}. Since we assume that $\frac{1}{2}\le \rho_0\le 2$, it is enough to prove that 
\[
 \sup_{B_{3/4}} |T-x|+ |T^{-1}-x|\les \lt(\int_{B_1} |T-x|^2 \rt)^{1/(d+2)}.
\]
We  first prove the estimate on $T$. Let $u(x):=T(x)-x$.  By monotonicity of $T$, for a.e. $x, y\in B_1$,
 \begin{equation}\label{monotony}(u(x)-u(y))\cdot(x-y)\ge -|x-y|^2.\end{equation}
Let $y\in B_{3/4}$ be such that \eqref{monotony} holds for a.e. $x\in B_1$. 
By translation we may assume that $y=0$. By rotation, it is enough to prove for the first coordinate of $u$ that 
 \[u_1(0)\les \lt(\int_{B_{1/4}} |u|^2\rt)^{1/(d+2)}.\]
Taking $y=0$ in \eqref{monotony}, we find for a.e. $x\in B_{1/4}$
\[u(0)\cdot x\le u(x)\cdot x+|x|^2\les |u(x)|^2+|x|^2.\] 
  Integrating the previous inequality over the ball $B_r(r e_1)$, we obtain
 \[ u(0)\cdot r e_1\les \dashint_{B_r(r e_1)} |u|^2 + r^2,\]
 so that 
 \[u_1(0)\les  \frac{1}{r^{d+1}} \int_{B_1} |u|^2  +r.\]
 Optimizing in $r$ yields \eqref{supestim}.
 We now prove the estimate on $T^{-1}$. By the above argument for $T$ in the ball $B_{4/5}$ instead of $B_{3/4}$,
 it is enough to show that $T^{-1}(B_{3/4})\subset B_{4/5}$. Assume that there exists $y\in B_{3/4}$ and $x\in \R^d$ with 
 $T(x)=y$ but $|x|\ge 4/5$. Let then $z\in \partial B_{\frac{1}{2}(\frac{3}{4}+\frac{4}{5})}\cap [x,y]$. By monotonicity of $T$,
 \begin{align*}
  0&\le (T(x)-T(z))\cdot(x-z)\\
  &=(y-z)\cdot(x-z)+(z-T(z))\cdot(x-z)\\
  &\le -\frac{1}{40} |x-z|+ |x-z| |T(z)-z|\\
  &\le |x-z|\lt(-\frac{1}{40} +\sup_{B_{\frac{1}{2}(\frac{3}{4}+\frac{4}{5})}} |T-x|\rt),
 \end{align*}
 which is absurd if $\E(\rho_0,\rho_1,T,1)\ll1$ by the $L^\infty$ bound on $T$ on the ball $B_{\frac{1}{2}(\frac{3}{4}+\frac{4}{5})}$.
\\
 
 Since  \eqref{Ttinclu} is a  direct consequence of \eqref{supestim}, we are left with the proof of \eqref{infestimTt}.
If $x\in \R^d$ is such that $T_t(x)\in B_{1/2}$, then by \eqref{supestim} in the form of $|T_t(0)|=o(1)$, where $o(1)$ denotes a function that goes
to zero as $\E(\rho_0,\rho_1,T,1)$ goes to zero,
\begin{align*}
 \frac{1}{4}(1+ o(1))&\ge |T_t(x)-T_t(0)|^2\\
 &=t^2|T(x)-T(0)|^2+ 2t(1-t) (T(x)-T(0))\cdot x +(1-t)^2|x|^2\\
 &\stackrel{\eqref{monotony}}{\ge}  t^2|T(x)-T(0)|^2 +(1-t)^2|x|^2\\
   & \  \ge \frac{1}{2}\min\lt\{|T(x)-T(0)|^2,|x|^2\rt\}.
\end{align*}
 From this we see that  $x$ or $T(x)$ is in $ B_{\frac{1}{\sqrt{2}}+o(1)}\subset B_{3/4}$. In the first case, \eqref{infestimTt} is proven while in the second, we have thanks to \eqref{supestim} that 
$x\in T^{-1}(T(x))\subset T^{-1}( B_{\frac{1}{\sqrt{2}}+o(1)})\subset B_{3/4}$ from which we  get \eqref{infestimTt} as well.
\end{proof}
Our second lemma is a localized version of McCann's displacement convexity (see \cite[Cor. 4.4]{McCann}).
\begin{lemma}
 Assume that $\rho_0(0)=\rho_1(0)=1$ and that $\E(\rho_0,\rho_1,T,1)+[\rho_0]_{\alpha, 1}+[\rho_1]_{\alpha,1}\ll 1$. Then for $t\in[0,1]$, it holds
 \begin{equation}\label{eq:displconv}
  \sup_{B_{1/2}} \rho(t,\cdot)\le 1+[\rho_0]_{\alpha, 1}+[\rho_1]_{\alpha,1}.
 \end{equation}
\end{lemma}
\begin{proof}
 We start by pointing out that since $\rho_0(0)=\rho_1(0)=1$ and $[\rho_0]_{\alpha, 1}+[\rho_1]_{\alpha,1}\ll1$ we have for $i=0,1$,
 \begin{equation}\label{suprhoi}
  \sup_{B_1}|1-\rho_i|\le [\rho_i]_{\alpha,1}\ll1.
 \end{equation}
 For every $t\in (0,1)$, the map 
 $T_t$ has a well-defined inverse $\rho(t,\cdot)-$a.e. (see the proof of \cite[Th. 8.1]{Viltop}) so that for $x\in B_{1/2}$, \eqref{eq:Jacob} can be written as
 \[
  \rho(t,x)=\frac{\rho_0(T_t^{-1}(x))}{\det \nabla T_t (T_t^{-1}(x))}.
 \]
By concavity of $\det(\cdot)^{1/d}$ on non-negative symmetric matrices, we have 
\[
 \det \nabla T_t (T_t^{-1}(x))\ge \lt(\det \nabla T(T_t^{-1}(x))\rt)^{t}.
\]
By \eqref{eq:Jacob}, $\det \nabla T(T_t^{-1}(x))=\frac{\rho_0(T_t^{-1}(x))}{\rho_1(T(T_t^{-1}(x)))}$, so that 
\[
\rho(t,x)\le \lt(\rho_0(T_t^{-1}(x)\rt)^{1-t}\lt(\rho_1(T(T_t^{-1}(x)))\rt)^{t}.
\]
Since $\E(\rho_0,\rho_1,T,1)\ll 1$ and \eqref{suprhoi} holds, by \eqref{infestimTt} and \eqref{supestim}, we have $T_t^{-1}(B_{1/2})\subset B_1$ and $T(T_t^{-1}(B_{1/2}))\subset B_1$. By \eqref{suprhoi}, we then have
\[
 \rho(t,x)\le\lt( 1+[\rho_0]_{\alpha,1}\rt)^{1-t} \lt(1+[\rho_1]_{\alpha, 1}\rt)^t, 
\]
which by Young's inequality concludes the proof of \eqref{eq:displconv}.
\end{proof}

We now can turn to the proof of Theorem \ref{theoepsintro}. We first prove that the deviation of the velocity field $v:=\frac{d j}{d \rho}$ from being the gradient of a harmonic function is locally controlled by the Eulerian energy. 
   The construction we use is somewhat reminiscent of the  Dacorogna-Moser construction (see \cite{Santam}).
\begin{proposition}\label{Prop:BBharm}
 Let $(\rho,j)$ be the minimizer of \eqref{BenamBre}. Assume that $\rho_0(0)=\rho_1(0)=1$ and that 
 \begin{equation}\label{hyp:small}
  \E(\rho_0,\rho_1,T,1)+[\rho_0]_{\alpha, 1}+ [\rho_1]_{\alpha,1}\ll 1.
 \end{equation}
Then, there exists $\phi$ harmonic in $B_{1/2}$
 and such that 
 \begin{equation}\label{eq:enerharm}
 \int_{B_{1/2}} |\nabla \phi|^2\les \int_{0}^1\int_{B_1} \frac{1}{\rho}|j|^2 +[\rho_0]^2_{\alpha, 1}+ [\rho_1]^2_{\alpha,1},
\end{equation}
and 
 \begin{equation}\label{eq:estimdistharmBB}
  \int_0^1\int_{B_{1/2}} \frac{1}{\rho}|j-\rho\nabla \phi|^2\les \lt(\int_0^1\int_{B_1} \frac{1}{\rho}|j|^2\rt)^{\frac{d+2}{d+1}}+[\rho_0]^2_{\alpha, 1}+ [\rho_1]^2_{\alpha,1}.
 \end{equation}

\end{proposition}
\begin{remark}
 The crucial point in \eqref{eq:estimdistharmBB} is that the right-hand side is strictly super-linear in $\int_0^1\int_{B_1} \frac{1}{\rho}|j|^2$, and at least quadratic in $[\rho_0]_{\alpha,1}+[\rho_1]_{\alpha,1}$. 
\end{remark}
\begin{proof}
Without loss of generality, we may assume that $\int_0^1\int_{B_1} \frac{1}{\rho}|j|^2\ll 1$ since otherwise we can take $\phi=0$. Notice that since $\rho_0(0)=\rho_1(0)=1$, thanks to \eqref{hyp:small}, 
if we let 
\[\gamma:=[\rho_0]_{\alpha, 1}+ [\rho_1]_{\alpha,1} \qquad \textrm{and } \qquad \delrho:=\rho_1-\rho_0,\]
we have by \eqref{eq:displconv}
 \begin{equation}\label{displconvgamma}
  \rho\le 1+\gamma
 \end{equation}
 and since $\rho_0(0)=\rho_1(0)=1$, 
\begin{equation}\label{eq:estimdelrho}
\sup_{B_1} |\delrho| \les [\rho_0]_{\alpha,1}+[\rho_1]_{\alpha,1}\le \gamma.\end{equation}
\\

{\it Step 1} [Choice of a good radius]
% Letting for $R>0$, $f:= j\cdot\nu$ be the normal component of $j$ on $\partial B_R$ and then $\Bf:=\int_0^1 f(x,t) dt$, we argue that we can find $R\in(1/2,1)$ such that 
% it is enough to find there exists $R\in(1/2,1)$ such that 
 Using  \eqref{displconvgamma}, and Fubini, we can find a radius $R\in (1/2,1)$ such that
 \begin{equation}\label{eq:hypf}
  \int_{\partial B_R} \int_0^1|j|^2\les \int_{B_1}\int_0^1 |j|^2\les \int_0^1\int_{B_1} \frac{1}{\rho}|j|^2
 \end{equation}
with the understanding that $R$ is a Lebesgue point of $r\mapsto j\in L^2(\partial B_r)$ with respect to  the weak topology. Notice in particular that \eqref{eq:hypf} implies that $j\in L^2(B_R)$.
%almost-every point of $\partial B_R\times (0,1)$ (w.\ r.\ t.\ the $d$-dimensional Hausdorff measure) 
%is a Lebesgue point of $j$.  
% We now let $\Bf(x):=\int_0^1 f(x,t) dt$, where for $(x,t)\in \partial B_R\times (0,1) $, $f(x,t):= j\cdot \nu$
%denotes the normal component. \\
%Before going further, let us make a small technical observation. 
We claim that for every function $\zeta\in H^1(B_R\times (0,1))$
\footnote{we consider here are larger class of test functions than $C^{1}(\overline{B_R}\times[0,1])$ 
since we want to apply \eqref{eq:localconteq} to the  function $\tilde \phi$ defined in \eqref{deftildephi}.},
\begin{equation}\label{eq:localconteq}
 \int_0^1\int_{B_R} \rho \partial_t\zeta+j\cdot \nabla \zeta=\int_0^1\int_{\partial B_R} \zeta f +\int_{B_R} \zeta(\cdot,1)\rho_1-\zeta(\cdot,0)\rho_0,
\end{equation}
where $f:= j\cdot \nu$ denotes the normal component of $j$.
%This is not obvious since $j$ is not \textit{a priori} regular enough to have a well-defined trace on $\partial (B_R\times(0,1))$.
%Up to extending $(\rho,j)$ to $B_1\times (-\delta,1+\delta)$ by $(1,0)$,
% we can assume that $(\rho,j)$ is smooth in a neighborhood of the top and bottom parts of the domain.
%We thus only need  to worry about the lateral parts of the boundary. 
To this purpose, for $0<\eps\ll 1$ we introduce the cut-off function 
\[\eta_\eps(x):=\begin{cases}
                 1 &\textrm{if } |x|\le R-\eps\\
                 \frac{R-|x|}{\eps} &\textrm{if } R-\eps\le |x|\le R\\
                 0 &\textrm{otherwise}
                \end{cases}\]
%    For $\zeta\in H^1(B_R)$, let us still denote by $\zeta$ an arbitrary extension with compact support. 
and obtain by admissibility of $(\rho,j)$
\begin{align*}
\int_{\R^2}\eta_\eps(\zeta(\cdot,1)\rho_1-\zeta(\cdot,0)\rho_0)
&=\int_0^1\int_{\R^2} \partial_t (\zeta \eta_\eps) \rho + \nabla (\zeta \eta_\eps)\cdot j\nonumber\\
&=\int_{0}^1\int_{\R^2}\eta_\eps \partial_t \zeta \rho+ \eta_\eps \nabla \zeta\cdot j
-\frac{1}{\eps}\int_{0}^1\int_{B_{R}\backslash B_{R-\eps}} \zeta j\cdot \nu.
\end{align*}
Letting $\eps$ go to zero and using the above Lebesgue-point property of $R$, we obtain \eqref{eq:localconteq}.\\               
\medskip

{\it Step 2} [Definition of $\phi$]
We will argue that it is enough to establish 
\begin{equation}\label{eq:topproveharmBB}
  \int_0^1\int_{B_R} \frac{1}{\rho}|j-\rho\nabla \tilde \phi|^2\les\lt(\int_0^1\int_{B_1} \frac{1}{\rho}|j|^2\rt)^{\frac{d+2}{d+1}}+\gamma^2,
\end{equation}
where $\tilde \phi$ is defined via
\begin{equation}\label{deftildephi}
 \begin{cases}
   -\Delta \tilde \phi=\delrho & \textrm{ in } B_R\\
   \frac{\partial \tilde \phi}{\partial \nu}=\Bf  &\textrm{ on } \partial B_R,
  \end{cases}
\end{equation}
with $\Bf:=\int_0^1 f dt$ and is estimated as 
\begin{equation}\label{eq:enertildephi}
  \int_{B_R} |\nabla \tilde  \phi|^2\les \int_0^1\int_{B_1} \frac{1}{\rho}|j|^2+\gamma^2.
\end{equation}
Moreover, defining for  $1\gg r>0$,  $A_r:= B_{R}\backslash B_{R(1-r)}$, we will show that  
\begin{equation}\label{eq:enerannulus}
 \int_{A_r} |\nabla \tilde \phi|^2\les  r\lt(\int_{\partial B_R} |\Bf|^2 +\gamma^2\rt).
\end{equation}

By \eqref{eq:localconteq} applied to $\zeta=1$, we get 
\begin{equation*}%\label{compatdeltarho}
 \int_{B_R} \delrho=-\int_{\partial B_R} \bar f,
\end{equation*}
so that \eqref{deftildephi} is indeed solvable. We decompose $\tilde \phi=\phi+\hat \phi$ with  $\phi$ and $\hat \phi$ solutions of 
\begin{equation}\label{defphi}\begin{cases}
   -\Delta \phi=0 &\textrm{in } B_R\\
   \frac{\partial \phi}{\partial \nu}=\Bf+\frac{1}{\H^{d-1}(\partial B_R)}\int_{ B_R} \delrho &\textrm{on } \partial B_R,
  \end{cases}\qquad \qquad 
  \begin{cases}
   -\Delta \hat \phi=\delrho &\textrm{in } B_R\\
   \frac{\partial \hat \phi}{\partial \nu}=-\frac{1}{\H^{d-1}(\partial B_R)}\int_{ B_R} \delrho &\textrm{on } \partial B_R,
  \end{cases}
  \end{equation}
Applying \eqref{energieestimphi} from Lemma \ref{lem:harm} (with the radius $1$ replaced by $R\sim 1$)
we have,
\[
 \int_{B_{1/2}} |\nabla \phi|^2\le \int_{B_R} |\nabla \phi|^2\les \int_{\partial B_R} |\Bf|^2+ \sup_{B_R}|\delrho|^2\stackrel{\eqref{eq:hypf}\&\eqref{eq:estimdelrho}}{\les}  \int_0^1\int_{B_1} \frac{1}{\rho}|j|^2+  \gamma^2,
\]
and thus  \eqref{eq:enerharm} holds.   Since by \eqref{Poi:energieestimphi} from Lemma \ref{lem:Poi} (with the radius $1$ replaced by $R\sim 1$),
\begin{equation}\label{eq:enerhatphi}
 \int_{B_R} |\nabla \hat \phi|^2\les \gamma^2,
\end{equation}
estimate \eqref{eq:enertildephi} is obtained by
\[
 \int_{B_R} |\nabla \tilde  \phi|^2\les \int_{B_R} |\nabla \phi|^2+\int_{B_R} |\nabla \hat \phi|^2\les \int_0^1\int_{B_1} \frac{1}{\rho}|j|^2+\gamma^2.
\]
Similarly, \eqref{eq:enerannulus} follows from  \eqref{estimaPoisannulus} and \eqref{Poi:estimaPoisannulus}. \\
Assume now that  \eqref{eq:topproveharmBB} is established. We then get \eqref{eq:estimdistharmBB}:
\begin{align*}
 \int_0^1\int_{B_R} \frac{1}{\rho}|j-\rho\nabla  \phi|^2&=\int_0^1\int_{B_R} \frac{1}{\rho}|j-\rho\nabla \tilde \phi+\rho \nabla \hat \phi|^2\\
 &\les \int_0^1\int_{B_R} \frac{1}{\rho}|j-\rho\nabla \tilde \phi|^2+ \int_0^1\int_{B_R}\rho|\nabla \hat \phi|^2\\
 % &\stackrel{\eqref{displconvgamma} \& \eqref{eq:topproveharmBB}}{\les} \lt(\int_{0}^1\int_{B_1} \frac{1}{\rho}|j|^2\rt)^{\frac{d+2}{d+1}} + \gamma \int_{0}^1\int_{B_1} \frac{1}{\rho}|j|^2+ \int_0^1\int_{B_R}|\nabla \hat \phi|^2\\
 &\stackrel{\eqref{eq:topproveharmBB} \& \eqref{displconvgamma} \&  \eqref{eq:enerhatphi}}{\les} \lt(\int_0^1\int_{B_1} \frac{1}{\rho}|j|^2\rt)^{\frac{d+2}{d+1}}+\gamma^2.
\end{align*}
\medskip

{\it Step 3} [Quasi-orthogonality]
We start the proof of \eqref{eq:topproveharmBB}. In order to keep notation light, we will assume from now on that $R=1/2$. Here we prove that 
\begin{equation}\label{secondreduc}
 \int_{0}^1\int_{B_{1/2}} \frac{1}{\rho}|j-\rho\nabla \tilde \phi|^2\le \int_{0}^1\int_{B_{1/2}} \frac{1}{\rho} |j|^2- (1-\gamma)\int_{B_{1/2}} |\nabla \tilde \phi|^2. 
\end{equation}
Notice that if $\rho=0$ then $j=0$ and thus also $j-\rho \nabla \tilde \phi=0$, so that the left-hand side of \eqref{secondreduc} is well defined (see the discussion  below \eqref{BenamBre}).
Based on this we compute
\begin{align*}
 \frac{1}{2}\int_{0}^1\int_{B_{1/2}} \frac{1}{\rho}|j-\rho\nabla \tilde \phi|^2&=\frac{1}{2} \int_{0}^1\int_{B_{1/2}} \frac{1}{\rho} |j|^2-\int_{0}^1\int_{B_{1/2}} j\cdot\nabla \tilde \phi+\frac{1}{2} \int_{0}^1\int_{B_1} \rho |\nabla \tilde \phi|^2\\
 &=\frac{1}{2} \int_{0}^1\int_{B_{1/2}} \frac{1}{\rho} |j|^2 -\int_{0}^1\int_{B_{1/2}} \lt(1- \frac{\rho}{2}\rt)|\nabla \tilde \phi|^2 -\int_{0}^1\int_{B_{1/2}} (j-\nabla \tilde\phi)\cdot\nabla \tilde \phi\\
 &\stackrel{\eqref{displconvgamma}}{\le} \frac{1}{2} \int_{0}^1\int_{B_{1/2}} \frac{1}{\rho} |j|^2 -\frac{1-\gamma}{2}\int_{0}^1\int_{B_{1/2}}|\nabla \tilde \phi|^2 -\int_{0}^1\int_{B_{1/2}} (j-\nabla \tilde \phi)\cdot\nabla \tilde \phi.
\end{align*}
 Using \eqref{eq:localconteq} with $\zeta=\tilde \phi$ and testing \eqref{deftildephi} with  $\tilde \phi$, we have 
\[
 \int_{0}^1\int_{B_{1/2}} (j-\nabla \tilde \phi)\cdot\nabla \tilde \phi=\int_{\partial B_{1/2}} \tilde \phi \lt(\int_0^1 f-\Bf\rt)=0, 
\]
where we recall that $\Bf=\int_0^1 f$. This proves \eqref{secondreduc}.\\
\medskip

{\it Step 4} [The main estimate]
In this last step, we establish that 
\begin{equation}\label{mainstep}
 \int_{0}^1\int_{B_{1/2}} \frac{1}{\rho} |j|^2- \int_{B_{1/2}} |\nabla \tilde \phi|^2\les   \lt(\int_0^1\int_{B_1} \frac{1}{\rho}|j|^2\rt)^{\frac{d+2}{d+1}}+\gamma^2.
\end{equation}
Thanks to \eqref{secondreduc} and \eqref{eq:enertildephi}, this would yield \eqref{eq:topproveharmBB}. 
 By minimality of $(\rho,j)$, it is enough to construct a competitor $(\tilde{\rho},\tilde{j})$ that agrees with $(\rho,j)$ outside of $B_{1/2}\times(0,1)$ and that satisfies the upper bound given through \eqref{mainstep}. 
%Up to reducing the class of competitors for $\rho$, we can assume that $\rho>1/2$ in $B_1\times[0,1]$. This class is not empty. Indeed,
% for $t\in(0,1)$, let $u_t$ be defined by
% \[ -\Delta u_t = -\frac{1}{|B_1|} \int_{\partial B_1} f \quad \textrm{ in } B_1 \qquad \textrm{and} \qquad \scp{\nabla u_t}{\nu}=f \quad \textrm{on } \partial B_1,\]
% and then $j^0=\nabla u_t$ and $\rho^0=1-\int_0^t \Div j^0(x,s) ds =1-\frac{1}{|B_1|}\int_{0}^t \int_{\partial B_1} f$, which are admissible if $\int_0^1\int_{\partial B_1} f^2$ is small enough.
 We now make the following ansatz 
\[(\tilde\rho,\tilde j):=\begin{cases}
                          (t\rho_1 +(1-t)\rho_0,\nabla \tilde \phi) &\textrm{ in }B_{1/2(1-r)}\times(0,1),\\
                          (t\rho_1 +(1-t)\rho_0+s,\nabla \tilde \phi+q) &\textrm{ in }A_r\times(0,1),
                         \end{cases}
                         \]
with $(s,q)\in \Lambda$, where $\Lambda$ is the set defined in Lemma \ref{lemLambda} with $f$ replaced by $f-\Bf$ and the radius $1$ replaced by $1/2$. Notice that if $|s|\le 1/2$, by \eqref{hyp:small} and $\rho_0(0)=\rho_1(0)=1$, 
\begin{equation}\label{boundstilderho}
 \frac{1}{4}\le \tilde \rho.
\end{equation}
Thanks to \eqref{deftildephi} for $\tilde \phi$, \eqref{localcontieq} for $(s,q)$ and  \eqref{eq:localconteq} for $(\rho,j)$, $(\tilde{\rho},\tilde{j})$ extended by $(\rho,j)$ outside of $B_{1/2}\times(0,1)$
is indeed admissible for \eqref{BenamBre}.\\
By Lemma \ref{lemLambda}, if $r\gg \lt(\int_0^1\int_{\partial B_{1/2}} (f-\Bf)^2\rt)^{1/(d+1)}$, we may choose $(s,q)\in \Lambda$ such that 
 \begin{equation}\label{eq:estimminLambdaproof}
 \int_0^1\int_{A_r}| q|^2\les r\int_0^1\int_{\partial B_{1/2}} (f-\Bf)^2.
 \end{equation}
 By definition of $(\tilde \rho, \tilde j)$,
 \begin{multline}\label{eq:estimgap}
 \int_0^1\int_{B_{1/2}} \frac{1}{\tilde{\rho}}|\tilde{j}|^2-\int_{B_{1/2}}|\nabla \tilde \phi|^2\le \int_0^1\int_{B_{1/2(1-r)}} \frac{1}{t\rho_1+(1-t)\rho_0}|\nabla \tilde \phi|^2 -\int_{ B_{1/2(1-r)}} |\nabla \tilde \phi|^2\\
 +\int_0^1\int_{A_r} \frac{1}{\tilde \rho}|\nabla \tilde \phi +q|^2.
 \end{multline}
The first two terms on the right-hand side can be estimated as
\begin{align}
 \int_0^1\int_{B_{1/2(1-r)}} \frac{1}{t\rho_1+(1-t)\rho_0}|\nabla \tilde \phi|^2 - |\nabla \tilde \phi|^2&=\int_0^1\int_{B_{1/2(1-r)}} \frac{t(1-\rho_0)+(1-t)(1-\rho_1)}{t\rho_1+(1-t)\rho_0}|\nabla \tilde \phi|^2\nonumber\\
 &\stackrel{\eqref{hyp:small}}{\les}\gamma\int_{B_{1/2}(1-r)} |\nabla \tilde \phi|^2\nonumber\\
 &\stackrel{\eqref{eq:enertildephi}}{\les} \gamma\lt(\int_0^1\int_{B_1} \frac{1}{\rho}|j|^2 +\gamma^2\rt) \label{eq:estimfirsttermgap}.
\end{align}
We  now estimate the last term of \eqref{eq:estimgap}:
 \begin{align*}\int_0^1\int_{A_r} \frac{1}{\tilde \rho}|\nabla \tilde \phi +q|^2
 &\stackrel{\eqref{boundstilderho}}{\les}  \int_{0}^1\int_{A_r}|\nabla \tilde \phi|^2+ |q|^2\\
 &\stackrel{\eqref{eq:estimminLambdaproof}}{\les}  \int_{A_r}|\nabla  \tilde  \phi|^2+  r\int_{\partial B_{1/2}}(f-\Bf)^2\\
 &\stackrel{\eqref{eq:enerannulus}}{\les}r\lt(\int_{\partial B_{1/2}} \Bf^2 +\gamma^2\rt).\end{align*}
 Taking $r$ to be a large but order-one multiple of 
\[
\lt(\int_0^1\int_{\partial B_{1/2}} (f-\Bf)^2\rt)^{1/(d+1)}\le\lt(\int_0^1\int_{\partial B_{1/2}} f^2\rt)^{1/(d+1)}
\stackrel{\eqref{eq:hypf}}{\les}\lt( \int_0^1\int_{B_1} \frac{1}{\rho}|j|^2\rt)^{1/(d+1)}
\]
 yields
\[
 \int_0^1\int_{A_r} \frac{1}{\tilde \rho}|\nabla \tilde \phi +q|^2\les  \lt(\int_0^1\int_{B_1} \frac{1}{\rho}|j|^2\rt)^{1/(d+1)}\lt(\int_0^1\int_{B_1} \frac{1}{\rho}|j|^2+\gamma^2\rt).
\]
Plugging this and  \eqref{eq:estimfirsttermgap} into \eqref{eq:estimgap},
\begin{multline*}
  \int_0^1\int_{B_{1/2}} \frac{1}{\tilde{\rho}}|\tilde{j}|^2-\int_{B_{1/2}}|\nabla \tilde \phi|^2\les \lt( \lt(\int_0^1\int_{B_1} \frac{1}{\rho}|j|^2\rt)^{1/(d+1)}+\gamma\rt)\lt( \int_0^1\int_{B_1} \frac{1}{\rho}|j|^2+\gamma^2\rt)\\
  \les \lt(\int_0^1\int_{B_1} \frac{1}{\rho}|j|^2\rt)^{\frac{d+2}{d+1}}+\gamma^2,
\end{multline*}
where we have used Young's inequality and the fact that $2>\frac{d+2}{d+1}$. This proves \eqref{mainstep}.
%\begin{multline*}
% \int_{0}^1\int_{B_{1/2}} \frac{1}{\tilde{\rho}} |\tilde{j}|^2- \int_{B_{1/2}} |\nabla \phi|^2\les \lt(\int_0^1\int_{\partial B_{1/2}} (f-\Bf)^2\rt)^{1/(d+1)} \int_{\partial B_{1/2}} \Bf^2 \\
% +\lt(\int_0^1\int_{\partial B_{1/2}} (f-\Bf)^2\rt)^{\frac{d+2}{d+1}}.
%\end{multline*}
% Using the triangle inequality, this gives \eqref{mainstep}.
 \end{proof}

\begin{remark}
 The quasi-orthogonality property \eqref{secondreduc} is a generalization of the following classical fact: If $\phi$
 is a harmonic function with $\frac{\partial \phi}{\partial \nu}=f$ on $\partial B_1$, then for every divergence-free
 vector-field $b$ with $b\cdot \nu =f$ on $\partial B_1$
 \[
  \int_{B_1}|b-\nabla \phi|^2=\int_{B_1} |b|^2-\int_{B_1}|\nabla \phi|^2,
 \]
so that the minimizers $b$ of the left-hand side coincide with the minimizers of the right-hand side. See for instance \cite[Lem. 2.2]{Merl} for an application of this idea in a different context. 
\end{remark}

We now prove that \eqref{eq:estimdistharmBB} implies a similar statement in the Lagrangian setting, namely that the distance of the 
displacement $T-x$ to the set of gradients of harmonic functions is (locally) 
controlled by the energy. This is reminiscent of the harmonic approximation property for minimal surfaces (see \cite[Sec. III.5]{Maggi}). 

\begin{proposition}\label{prop:Lagestim}
 Let $T$ be the minimizer of \eqref{prob:OT} and assume that $\rho_0(0)=\rho_1(0)=1$. Then there exists  a  function $\phi$ harmonic in $B_{1/8}$,  such that 
 \begin{equation}\label{eq:estimdistharmLag}
  \int_{B_{1/8}} |T-(x+\nabla \phi)|^2 \rho_0\les \E(\rho_0,\rho_1,T,1)^{\frac{d+2}{d+1}}+[\rho_0]_{\alpha,1}^2+[\rho_1]_{\alpha,1}^2
 \end{equation}
 and 
 \begin{equation}\label{eq:energieestimphiLag}
  \int_{B_{1/8}} |\nabla \phi|^2\les \E(\rho_0,\rho_1,T,1)+[\rho_0]_{\alpha,1}^2+[\rho_1]_{\alpha,1}^2.
 \end{equation}

\end{proposition}
\begin{proof} To lighten notation, let $\E:=\E(\rho_0,\rho_1,T,1)$.
 Notice first that we may assume that $\E+[\rho_0]_{\alpha,1}^2+[\rho_1]_{\alpha,1}^2\ll 1$ since otherwise we can take $\phi=0$. \\ 
 
We  recall the definitions of the measures
\[
 \rho(\cdot,t):=T_t\sharp \rho_0 \qquad \textrm{and} \qquad j(\cdot,t):= T_t\sharp \lt[(T-Id)\rho_0\rt].
\]
We note that the velocity field $v=\frac{dj}{d\rho}$ satisfies  $v(T_t(x),t)=T(x)-x$  for a.e. $x\in \supp \rho_0$ (this can be seen arguing  
for instance as in the proof of  \cite[Th. 8.1]{Viltop}). Hence, by definition of the expression $\frac{1}{\rho}|j|^2$ and that of $\rho$, 
 \[ \int_0^1\int_{B_{1/2}}\frac{1}{\rho}|j|^2 = \int_0^1\int_{B_{1/2}} |v|^2 d\rho 
 =\int_0^1\int_{T_t^{-1}(B_{1/2})} |T-x|^2\rho_0\stackrel{\eqref{infestimTt}}{\les} \int_{B_1} |T-x|^2\rho_0= \E.
 \]
By Proposition \ref{Prop:BBharm}, we infer that there exists  a function $\phi$ harmonic in
 $B_{1/4}$ such that
 \begin{equation}\label{eq:estimLagBB}
  \int_0^1\int_{B_{1/4}}\frac{1}{\rho}|j-\rho \nabla \phi|^2\les \E^{\frac{d+2}{d+1}}+[\rho_0]_{\alpha,1}^2+[\rho_1]_{\alpha,1}^2\qquad \textrm{and} \qquad \int_{B_{1/4}}|\nabla \phi|^2 \les \E+[\rho_0]_{\alpha,1}^2+[\rho_1]_{\alpha,1}^2.
 \end{equation}

We now prove \eqref{eq:estimdistharmLag}.
 By the triangle inequality we have 
\[
  \int_{B_{1/8}} |T-(x+\nabla \phi)|^2\rho_0 \les  \int_{0}^1\int_{B_{1/8}} |T-(x+\nabla \phi \circ T_t)|^2\rho_0+\int_0^1\int_{B_{1/8}} |\nabla \phi-\nabla \phi \circ T_t|^2\rho_0.
\]
Using that for $t\in [0,1]$, $|T_t(x)-x|\le |T(x)-x|$, the second term on the right-hand side is estimated as above:
\begin{align*}
 \int_0^1\int_{B_{1/8}} |\nabla \phi-\nabla \phi \circ T_t|^2\rho_0&\stackrel{\eqref{Ttinclu}}{\les} \sup_{B_{3/16}}|\nabla^2 \phi|^2 \int_0^1\int_{B_{1/8}} |T_t-x|^2\rho_0\\
 &\stackrel{\eqref{Schauder}}{\les} \E \int_{B_{1/4}} |\nabla \phi|^2 \\
 &\stackrel{\eqref{eq:estimLagBB}}{\les} \E \lt(\E +[\rho_0]_{\alpha,1}^2+[\rho_1]_{\alpha,1}^2\rt).
\end{align*}
We thus just need to estimate the first term. Recall that $v=\frac{dj}{d\rho}$ satisfies $v(T_t(x),t)=T(x)-x$, so that we obtain for the integrand
$T(x)-(x+\nabla \phi(T_t(x))$ $=(v(t,\cdot)-\nabla\phi)(T_t(x))$ for a.e. $x\in \supp \rho_0$. Hence, by definition of $\rho$ and
by our convention on how to interpret $\frac{1}{\rho} |j-\rho \nabla \phi|^2$ when $\rho$ vanishes,
\begin{align*}
  \int_{0}^1\int_{B_{1/8}} |T-(x+\nabla \phi \circ T_t)|^2\rho_0&=\int_0^1\int_{T_t(B_{1/8})}| v-\nabla \phi|^2 d\rho\\
  &= \int_0^1\int_{T_t(B_{1/8})} \frac{1}{\rho} |j-\rho \nabla \phi|^2\\
  &\stackrel{\eqref{Ttinclu}}{\le} \int_0^1\int_{B_{1/4}} \frac{1}{\rho} |j-\rho \nabla \phi|^2\stackrel{\eqref{eq:estimLagBB}}{\les} \E^{\frac{d+2}{d+1}}+[\rho_0]_{\alpha,1}^2+[\rho_1]_{\alpha,1}^2.
\end{align*}

\end{proof}

Analogously to De Giorgi's proof of regularity for minimal surfaces (see for instance \cite[Chap. 25.2]{Maggi}), 
we are going to prove an ``excess improvement by tilting''-estimate. 
By this we mean that if at a certain scale $R$, the map $T$ is close to the identity, i.e. if $\E(\rho_0,\rho_1,T,R)+R^{2\alpha}([\rho_0]_{\alpha,1}^2+[\rho_1]_{\alpha,1}^2)\ll1$, 
then on a scale $\theta R$, after an affine change of coordinates, it is even closer to be the identity.  
Together with \eqref{eq:estimdistharmLag} from Proposition \ref{prop:Lagestim}, the main ingredient
of the proof are the regularity estimates \eqref{Schauder} from Lemma \ref{lem:harm} for harmonic functions.  

\begin{proposition}\label{iter}
 For every $\alpha'\in (0,1)$ there exist $\theta=\theta(d,\alpha, \alpha')>0$ and $C_\theta(d,\alpha,\alpha')>0$ with the property that for every $R>0$ such that $\rho_0(0)=\rho_1(0)=1$  and
 \begin{equation}\label{hyp:Erho}\E(\rho_0,\rho_1,T,R)+R^{2\alpha}([\rho_0]_{\alpha,R}^2+[\rho_1]_{\alpha,R}^2)\ll 1,\end{equation}
 there exist a symmetric matrix $B$ with $\det B=1$ and a vector $b$ with 
\begin{equation}\label{boundBc}|B-Id|^2 +\frac{1}{R^2} |b|^2\les\E(\rho_0,\rho_1,T,R)+R^{2\alpha}([\rho_0]_{\alpha,R}^2+[\rho_1]_{\alpha,R}^2),\end{equation}
such that, letting $\lambda:=\rho_1(b)^{\frac{1}{d}}$, $\hat x:=B^{-1} x$,  $\hat y:= \lambda B(y-b)$ and then 
\begin{equation}\label{defhat} 
\hat{T}(\hat x):=\lambda B(T(x)-b), \quad \hat{\rho}_0(\hat x):= \rho_0(x) \quad  \textrm{ and } \quad \hat \rho_1(\hat y):=\lambda^{-d} \rho_1(y),
\end{equation}
it holds $\hat \rho_0(0)=\hat \rho_1(0)=1$,
\begin{equation}\label{boundlambda}
|\lambda-1|^2\les  \E(\rho_0,\rho_1,T,R)+R^{2\alpha}[\rho_1]^2_{\alpha,R},
\end{equation}
and 
\begin{equation}\label{improvementE}
 \E(\hat \rho_0,\hat \rho_1,\hat{T},\theta R)\le \theta^{2\alpha'} \E(\rho_0,\rho_1,T,R)+C_\theta R^{2\alpha}\lt([\rho_0]_{\alpha,R}^2+[\rho_1]_{\alpha,R}^2\rt).
\end{equation}
\end{proposition}
\begin{proof}
By a rescaling $\tilde{x}= R^{-1}x$, which amounts to the re-definition $\tilde{T}(\tilde x):=R^{-1} T(R\tilde x)$ (which preserves optimality)
and $\tilde b:= R^{-1} b$, we may assume that $R=1$. \\
As in the previous proof, we will let $\E:=\E(\rho_0,\rho_1,T,1)$.
Let $\phi$ be the harmonic function given by Proposition \ref{prop:Lagestim} and then define $b:=\nabla \phi(0)$, $A:= \nabla^2\phi(0)$ and set $B:=e^{-A/2}$, so that $\det B=1$. 
 Using \eqref{Schauder} from Lemma \ref{lem:harm} and \eqref{eq:energieestimphiLag}
from Proposition \ref{prop:Lagestim}, 
we see that \eqref{boundBc} is satisfied. By definition of $\lambda$ and since $\rho_1(0)=1$, and $[\rho_0]_{\alpha,1}+[\rho_1]_{\alpha,1}\les 1$,
\[
 |\lambda-1|^2= |\rho_1(b)-1|^2\le |b|^{2\alpha} [\rho_1]^2_{1,\alpha}\stackrel{\eqref{boundBc}}{\les} \lt(\E^{\alpha}+1\rt)[\rho_1]^2_{\alpha,1}.
\]
Using Young's inequality with $p=\alpha^{-1}$ and $q=(1-\alpha)^{-1}$  we obtain for $\delta>0$,
\begin{equation}\label{estimlambdaeps}
 |\lambda-1|^2\le \delta \E+\frac{C_\alpha}{\delta} [\rho_1]^2_{\alpha,1}, 
\end{equation}
where $C_\alpha$ is a constant which depends only on $\alpha$. In particular, taking $\delta=1$ we obtain \eqref{boundlambda}.

Defining $\hat \rho_i$ and $\hat T$ as in \eqref{defhat}  we have by  (\ref{boundBc}) and \eqref{hyp:Erho}
\begin{align*}
\dashint_{B_{\theta}} |\hat{T}-\hat{x}|^2 \hat \rho_0&= \dashint_{B B_{\theta}} |\lambda B (T-b)-B^{-1} x|^2 \rho_0\\
&\les\lambda^2 \dashint_{B_{2\theta}} | T-(B^{-2} x+b)|^2\rho_0+ |1-\lambda|^2\dashint_{B_{2\theta}}|B^{-1}x|^2\rho_0\\
&\les\dashint_{B_{2\theta}} | T-(B^{-2} x+b)|^2\rho_0+\theta^2\lt(\theta^2 \E+\theta^{-2} [\rho_1]_{\alpha,1}^2\rt),
\end{align*}
where in the last line we used \eqref{estimlambdaeps} with $\delta=\theta$ and the fact that $\rho_0\les 1$ on $B_1$.
We split the first term on the right-hand side into three terms
\begin{align*}
\lefteqn{\dashint_{B_{2\theta}} | T-(B^{-2} x+b)|^2\rho_0}\nonumber\\
%where in the last line we used \eqref{boundBc}.
%Since by \eqref{boundBc} $B B_\theta\subset B_{2 \theta}$.
%Recalling $B=e^{-A/2}$ and $A=\nabla^2\phi(0)$, we obtain
%\begin{align*}\theta^{-2}\dashint_{B_{\theta}} |\hat{T}-x|^2 
%&\les \theta^{-2}\dashint_{ B_{2\theta}} | T-(e^{A} x+b)|^2\\
&\les \dashint_{ B_{2\theta}}|T-(x+\nabla \phi)|^2\rho_0+\dashint_{ B_{2\theta}}|(B^{-2}-Id-A)x|^2\rho_0
+\dashint_{ B_{2\theta}}|\nabla \phi-b-Ax|^2\rho_0\\
&\les \dashint_{ B_{2\theta}}|T-(x+\nabla \phi)|^2\rho_0+ \theta^2|B^{-2}-Id-A|^2
+\sup_{ B_{2\theta}}|\nabla \phi-b-Ax|^2.
\end{align*}
Recalling $B=e^{-A/2}$, $A=\nabla^2\phi(0)$, and $b=\nabla\phi(0)$, we obtain
\begin{align*}
\theta^{-2}\dashint_{B_{\theta}} |\hat{T}-x|^2\hat \rho_0
%& \qquad +\theta^{-2}\dashint_{ B_{2\theta}}|\nabla \phi-b-Ax|^2\\
&\stackrel{\eqref{eq:estimdistharmLag}}{\les} \theta^{-(d+2)}\lt(\E^{\frac{d+2}{d+1}}+[\rho_0]_{\alpha,1}^2+[\rho_1]_{\alpha,1}^2\rt) +|\nabla^2 \phi(0)|^4+\theta^2 \sup_{B_{2\theta}} |\nabla^3 \phi|^2\\
&\qquad \qquad +\theta^2 \E+\theta^{-2} [\rho_1]_{\alpha,1}^2\\
&\stackrel{\eqref{Schauder} \& \eqref{eq:energieestimphiLag}}{\les} \theta^{-(d+2)}\E^{\frac{d+2}{d+1}}+  \lt(\E+[\rho_0]_{\alpha,1}^2+[\rho_1]_{\alpha,1}^2\rt)^2+\theta^2\lt(\E+[\rho_0]_{\alpha,1}^2+[\rho_1]_{\alpha,1}^2\rt)\\
&\qquad \qquad  +\theta^2 \E +\theta^{-2} \lt([\rho_0]_{\alpha,1}^2+[\rho_1]_{\alpha,1}^2\rt).
 \end{align*}
 Since $\frac{d+2}{d+1}<2$ and $\E+[\rho_0]_{\alpha,1}^2+[\rho_1]_{\alpha,1}^2\ll \theta^2$ (recall that $\theta$ has not been fixed yet), this simplifies to  
 \begin{equation}\label{estimT}
  \theta^{-2}\dashint_{B_{\theta}} |\hat{T}-x|^2 \hat \rho_0\les \theta^{-(d+2)}\E^{\frac{d+2}{d+1}}+ \theta^2 \E+ \theta^{-2} \lt([\rho_0]_{\alpha,1}^2+[\rho_1]_{\alpha,1}^2\rt).
 \end{equation}
We may thus find a constant $C(d,\alpha)>0$ such that 
\[
 \theta^{-2}\dashint_{B_\theta} |\hat{T}-x|^2\hat{\rho}_0\le C \lt(\theta^{-(d+2)}\E^{\frac{d+2}{d+1}} +\theta^2 \E\rt)+ \theta^{-2}\lt([\rho_0]_{\alpha,1}^2+[\rho_1]_{\alpha,1}^2\rt).
\]
 We now fix $\theta(d,\alpha,\alpha')$ such that $C\theta^2 \le \frac{1}{2} \theta^{2\alpha'}$, which is possible because $\alpha'<1$. If $\E$ is small enough,
 $C\theta^{-(d+2)}\E^{\frac{d+2}{d+1}}\le\frac{1}{2} \theta^{2\alpha'} \E$ and thus 
\[ \theta^{-2}\dashint_{B_\theta} |\hat{T}-x|^2\hat{\rho}_0\le \theta^{2\alpha'}\E+\theta^{-2}\lt([\rho_0]_{\alpha,1}^2+[\rho_1]_{\alpha,1}^2\rt).\]
\end{proof}
Equipped with the one-step-improvement of Proposition \ref{iter}, the next proposition is the outcome of a Campanato iteration
(see for instance \cite[Chap. 5]{Giaquinta} for an application of Campanato iteration to obtain Schauder theory for linear elliptic systems).
It is a Campanato iteration on the  $C^{1,\alpha}$ level for the transportation map $T$
and thus proceeds by comparing $T$ to affine maps. The main ingredient is the {\it affine} invariance of transportation.
Proposition \ref{P3} amounts to an $\eps$-regularity result.
\begin{proposition}\label{P3}
 Assume that $\rho_0(0)=\rho_1(0)=1$ and that 
 \begin{equation}\label{hyp:ErhoR}
 \E(\rho_0,\rho_1,T,2R)+R^{2\alpha}([\rho_0]_{\alpha, 2R}^2+[\rho_1]_{\alpha,2R}^2)\ll 1, 
 \end{equation}
 then $T$ is of class $C^{1,\alpha}$ in $B_{R}$, with
 \[
  [\nabla T]_{\alpha,R}\les R^{-\alpha}\E(T,2R)^{1/2}+ [\rho_0]_{\alpha,2R}+[\rho_1]_{\alpha,2R}.
 \]

\end{proposition}

\begin{proof}
By Campanato's theory, see \cite[Th. 5.I]{campanato}, we have to prove that \eqref{hyp:ErhoR} implies
\begin{equation}\label{toproveC1alpha}
 \sup_{x_0\in B_R} \sup_{r\le \frac{1}{2}R} \min_{A,b} \frac{1}{r^{2(1+\alpha)}} \dashint_{B_r(x_0)}|T-(Ax+b)|^2 \les R^{-2\alpha} \E(T,2R)+ [\rho_0]^2_{\alpha, 2R}+[\rho_1]^2_{\alpha, 2R}. 
\end{equation}
Let us first notice that \eqref{hyp:ErhoR} implies that for every $x_0\in B_R$
 \begin{equation}\label{Ex0}
  \E:=R^{-2}\dashint_{B_R(x_0)} |T-x|^2\rho_0 \ll 1 \qquad \textrm{and} \qquad R^{\alpha}\lt( [\rho_0]_{\alpha,2R}+[\rho_1]_{\alpha,2R}\rt)\ll1 .
 \end{equation}
 Therefore, in order to prove \eqref{toproveC1alpha}, it is enough to prove that \eqref{Ex0} implies that for $r\le \frac{1}{2}R$,
 \begin{equation}\label{toproveEx0}
  \min_{A,b} \frac{1}{r^{2}} \dashint_{B_r(x_0)}|T-(Ax+b)|^2\les r^{2\alpha}\lt(R^{-2\alpha}\E +[\rho_0]_{\alpha,2R}^2+[\rho_1]_{\alpha,2R}^2\rt).
 \end{equation}
 Replacing $\rho_0$ by $\rho_0(x_0)^{-1}\rho_0$ and $\rho_1$ by $\rho_1(x_0)^{-1}\rho_1(x_0+\lt(\frac{\rho_0(x_0)}{\rho_1(x_0)}\rt)^{\frac{1}{d}}(\cdot-x_0))$ and thus
 $T$ by $x_0+\lt(\frac{\rho_1(x_0)}{\rho_0(x_0)}\rt)^{\frac{1}{d}}(T-x_0)$ which still satisfies \eqref{Ex0} thanks to $\rho_0(0)=\rho_1(0)=1$ and \eqref{hyp:ErhoR}, we may assume that $\rho_0(x_0)=\rho_1(x_0)=1$. 
Without loss of generality we may thus assume that $x_0=0$. 

Fix from now on an $\alpha'\in(\alpha,1)$. Thanks to \eqref{Ex0}, Proposition \ref{iter} applies and 
there exist a (symmetric) matrix $B_1$ of unit determinant, a vector $b_1$ and a positive number $\lambda_1$ such that $T_1(x):= \lambda_1 B_1(T(B_1x)-b_1)$, $\rho_0^1(x):=\rho_0(B_1 x)$ and $\rho_1^1(x):=\lambda_1^{-d}\rho_1(\lambda_1^{-1}B_1^{-1} x +b_1)$ satisfy
 \begin{equation}\label{estimEk1}\E_1:=\E(\rho_0^1,\rho_1^1,T_1,\theta R)\le \theta^{2\alpha'} \E+C_\theta R^{2\alpha}\lt([\rho_0]^2_{\alpha, R}+[\rho_1]^2_{\alpha, R}\rt).\end{equation}
 If $T$ is a  minimizer of \eqref{prob:OT}, then so is  $T_1$ with $(\rho_0,\rho_1)$ replaced by $(\rho_0^1,\rho_1^1)$. Indeed,  because $\det B_1=1$, $T_1$ sends $\rho_0^1$ on $\rho_1^1$  and if $T$ is the gradient of a convex function $\psi$
 then  $T_1=\nabla \psi_1$ where $\psi_1(x):=\lambda_1(\psi(B_1x)-b_1  \cdot B_1 x)$ is also a convex function, which characterizes optimality \cite[Th. 2.12]{Viltop}. 
 Moreover, for $i=0,1$
 \begin{equation}\label{preservboundHolder}
  [\rho_i^1]_{\alpha, \theta R}\le \lt(1+C(\E^{1/2}+R^\alpha [\rho_0]_{\alpha,R}+R^\alpha [\rho_0]_{\alpha,R})\rt)[\rho_i]_{\alpha,R}.
\end{equation}
Indeed, (we argue only for $\rho_1^1$ since the proof for $\rho_0^1$ is simpler), using that $\lambda_1^{-1} B_1^{-1}B_{\theta R} +b_1\subset B_R$ by \eqref{boundBc} and \eqref{boundlambda},
\begin{align*}
 [\rho_1^1]_{\alpha, {\theta R}}&=\lambda_1^{-d} \sup_{x,y\in B_{\theta R}} \frac{|\rho_1(\lambda_1^{-1}B_1^{-1}x+b_1)-\rho_1(\lambda_1^{-1}B_1^{-1}y+b_1)|}{|x-y|^\alpha}\\
 &=\lambda_1^{-d} \sup_{x,y\in B_{\theta R}} \frac{|\rho_1(\lambda_1^{-1}B_1^{-1}x+b_1)-\rho_1(\lambda_1^{-1}B_1^{-1}y+b_1)|}{|\lambda_1 B_1[ (\lambda_1^{-1}B_1^{-1}x+b_1)-(\lambda_1^{-1}B_1^{-1}y+b_1)]|^\alpha}\\
 &\le\lambda_1^{-d}|\lambda_1^{-1}B_1^{-1}|^\alpha\sup_{x,y\in B_{\theta R}} \frac{|\rho_1(\lambda_1^{-1}B_1^{-1}x+b_1)-\rho_1(\lambda_1^{-1}B_1^{-1}y+b_1)|}{| (\lambda_1^{-1}B_1^{-1}x+b_1)-(\lambda_1^{-1}B_1^{-1}y+b_1)|^\alpha}\\
 &\le \lambda_1^{-(d+\alpha)}|B_1^{-1}|^\alpha\sup_{x',y'\in B_{ R}} \frac{|\rho_1(x')-\rho_1(y')|}{| x'-y'|^\alpha}\\
 &=\lambda_1^{-(d+\alpha)}|B_1^{-1}|^\alpha[\rho_1]_{\alpha,R}.
\end{align*}
By \eqref{boundBc} and \eqref{boundlambda}, we get \eqref{preservboundHolder}. \\
 Therefore, we may iterate  Proposition \ref{iter} $K>1$ times
 to find a sequence of (symmetric) matrices $B_k$ with $\det B_k=1$, 
 a sequence of vectors $b_k$, a sequence of real numbers $\lambda_k$ and a sequence of maps $T_k$ such that setting for $1\le k\le K$,
 \[T_k(x):=\lambda_k B_k(T_{k-1}(B_k x)-b_k), \quad \rho_0^k(x):=\rho_0^{k-1}(B_k x) \quad \textrm{and} \quad \rho_1^{k}(x):=\lambda_k^{-d}\rho_1^{k-1}(\lambda_k^{-1}B_k^{-1} x +b_k),\] 
 it holds $\rho_0^k(0)=\rho_1^k(0)=1$ and 
 \begin{align}
 \E_k:=\E(\rho_0^k,\rho_1^k,T_k,\theta^k R)&\le \theta^{2\alpha' }\E_{k-1}+C_\theta \theta^{2(k-1)\alpha} R^{2\alpha} \lt([\rho_0^{k-1}]^2_{\alpha,\theta^{k-1}R}+[\rho_1^{k-1}]^2_{\alpha,\theta^{k-1}R}\rt),\label{estimEk}\\
 |B_k-Id|^2&\les\E_{k-1} +\theta^{2k\alpha} R^{2\alpha} \lt([\rho_0^{k-1}]^2_{\alpha,\theta^{k-1}R}+[\rho_1^{k-1}]^2_{\alpha,\theta^{k-1}R}\rt), \label{estimBk}\\
 \frac{1}{(\theta^{k-1} R)^2} |b_k|^2&\les \E_{k-1} +\theta^{2k\alpha} R^{2\alpha} \lt([\rho_0^{k-1}]^2_{\alpha,\theta^{k-1}R}+[\rho_1^{k-1}]^2_{\alpha,\theta^{k-1}R}\rt),\label{estimbk}\\
  |\lambda_k-1|^2&\les \E_{k-1} + \theta^{2k\alpha}R^\alpha [\rho^{k-1}_1]^2_{\alpha,\theta^{k-1}R}.\label{estimlambdak}
 \end{align}
Arguing as for \eqref{preservboundHolder}, we obtain that there exists $C_1(d,\alpha)>0$ such that 
\begin{equation}\label{preservboundHolderk}
 [\rho_i^k]_{\alpha, {\theta^k R}}\le \lt(1+C_1(\E_{k-1}^{1/2}+R^\alpha \theta^{k\alpha}\lt([\rho^{k-1}_0]_{\alpha,{\theta^{k-1}R}}+ [\rho^{k-1}_0]_{\alpha,{\theta^{k-1}R}}\rt)\rt)[\rho^{k-1}_i]_{\alpha, \theta^{k-1} R}.
\end{equation}

Let us prove by induction that the above  together with \eqref{Ex0} implies  that there exists $C_2(d,\alpha,\alpha')>0$ such that  for every $1\le k\le K$, 
\begin{equation}\label{hypinduction}
[\rho^k_i]_{\alpha, \theta^k R}\le (1+\theta^{k\alpha})[\rho^{k-1}_i]_{\alpha,{\theta^{k-1}R}}, \quad \theta^{-2k\alpha }\E_{k}\le C_2\lt( \E+R^{2\alpha}[\rho_0]^2_{\alpha, R}+R^{2\alpha}[\rho_1]^2_{\alpha, R}\rt). 
\end{equation}
 This will show that we can keep on iterating Proposition \ref{iter}.\\

By \eqref{estimEk1} and \eqref{preservboundHolder}, \eqref{hypinduction} holds for $K=1$ provided $C_2\ges C_\theta \theta^{-2\alpha}$. Assume that it holds for $K-1$. Let us start by the first part of \eqref{hypinduction}. Notice that the induction hypothesis implies that 
\begin{equation}\label{estimHolderrhoi}
 [\rho^{K-1}_i]_{\alpha, \theta^{K-1} R}\le \prod_{k=1}^{K-2} (1+\theta^{k\alpha}) [\rho_i]_{\alpha,R}\le C_3 [\rho_i]_{\alpha,R},
\end{equation}
where $C_3:=\prod_{k=1}^\infty (1+\theta^{k\alpha})<\infty$.  From   \eqref{hypinduction} and  \eqref{estimHolderrhoi} for $k=K-1$ we learn that we may choose the implicit small constant in \eqref{Ex0} such that we have 
\[
 C_1\lt( \theta^{-\alpha} \lt(\sup_{1\le k\le K-1} \theta^{-2k \alpha } \E_k\rt)^{1/2} +R^\alpha \sup_{1\le k\le K-1} [\rho_0^k]_{\alpha, \theta^{k} R}+[\rho_1^k]_{\alpha, \theta^{k} R}\rt)\le 1.
\]
Plugging this into \eqref{preservboundHolderk}, we obtain the first part of \eqref{hypinduction} for $k=K$. \\
Let us now turn to the second part of \eqref{hypinduction}. Dividing \eqref{estimEk} by $\theta^{2k\alpha}$ and taking the $\sup$ over $k\in [1,K]$, we obtain by \eqref{estimHolderrhoi},
\[
 \sup_{1\le k\le K} \theta^{-2k \alpha} \E_k\le \theta^{2(\alpha'-\alpha)}( \E+ \sup_{1\le k\le K-1} \theta^{-2k \alpha} \E_k )+C_\theta C_3^2R^{2\alpha}\lt([\rho_0]^2_{\alpha, R}+[\rho_1]^2_{\alpha, R}\rt).
\]
Since $\alpha'>\alpha$, $\theta^{2(\alpha'-\alpha)}<1$ and thus
\[
\sup_{1\le k\le K} \theta^{-2k \alpha} \E_k\le ( 1-\theta^{2(\alpha'-\alpha)})^{-1}\lt[\E+C_\theta C_3^2R^{2\alpha}\lt([\rho_0]^2_{\alpha, R}+[\rho_1]^2_{\alpha, R}\rt)\rt].
\]
Choosing $C_2:=( 1-\theta^{2(\alpha'-\alpha)})^{-1}\max\lt\{1, C_\theta C_3^2\rt\}$ we see that also the second part of \eqref{hypinduction} holds for $k=K$. \\

 Letting $\Lambda_k:=\prod_{i=1}^k\lambda_i$, $A_k:=B_kB_{k-1}\cdots B_1$  and $d_k:= \sum_{i=1}^k (\lambda_k B_k)(\lambda_{k-1} B_{k-1})\cdots (\lambda_iB_i) b_i$, we see that  $T_k(x)=\Lambda_k A_kT( A_k^*x)-d_k$. By \eqref{estimBk}, \eqref{hypinduction} and \eqref{estimHolderrhoi},  
\begin{equation}\label{estimAk}|A_k-Id|^2\les  \E +R^{2\alpha} [\rho_0]^2_{\alpha, R}+R^{2\alpha} [\rho_1]^2_{\alpha, R}\ll1,\end{equation}
so that  $B_{\frac{1}{2}\theta^kR}\subset A_k^* (B_{\theta^k R})$. By the same reasoning, we obtain from \eqref{estimlambdak},
\begin{equation}\label{estimlambdakprod}
 |\Lambda_k-1|\ll1.
\end{equation}
We then conclude by definition of $T_k$ that 
\begin{align*}
 \min_{A,b} \frac{1}{(\frac{1}{2}\theta^{k}R)^2} \dashint_{B_{\frac{1}{2}\theta^kR}}|T-(Ax+b)|^2&\les \frac{1}{({\theta}^{k} R)^2} \dashint_{A_k^*(B_{{\theta}^k R})}|T-\Lambda_k^{-1}A_k^{-1}A_k^{-*}x+\Lambda_k^{-1}A_k^{-1}d_k)|^2\\
 &=\frac{1}{({\theta}^{k} R)^2} \dashint_{B_{{\theta}^k R}}|A_k^{-1}\Lambda_k^{-1}(T_k-x)|^2 \\
 &\stackrel{\eqref{estimAk}\&\eqref{estimlambdakprod}}{\les} \frac{1}{({\theta}^{k} R)^2} \dashint_{B_{{\theta}^k R}}|T_k-x|^2\\
 &\stackrel{\eqref{hypinduction}}{\les} \theta^{2k \alpha}\lt(\E+R^{2\alpha}[\rho_0]^2_{\alpha, R}+R^{2\alpha}[\rho_1]^2_{\alpha, R}\rt).
\end{align*}
From this \eqref{toproveEx0} follows, which concludes the proof of \eqref{toproveC1alpha}.
 
\end{proof}

With this $\eps$-regularity result at hand, we  now may prove that $T$ is a $C^{1,\alpha}$ diffeomorphism outside of a set of measure zero.
\begin{theorem}\label{theo:final}
 For $E$ and $F$ two bounded open sets, let $\rho_0:E\to \R^+$ and $\rho_1:F\to \R^+$ be two $C^{0,\alpha}$  densities with equal masses, both bounded  and bounded away from zero and let $T$ be the minimizer of \eqref{prob:OT}. 
 There exist
 open sets $E'\subset E $ and $F'\subset F$
 with $|E\backslash E'|=|F\backslash F'|=0$ and such that $T$ is a $C^{1,\alpha}$ diffeomorphism between $E'$ and $F'$.
\end{theorem}
\begin{proof}
By  the Alexandrov Theorem \cite[Th. 14.25]{VilOandN}, 
there exist two sets of full measure $E_1\subset E$ and $F_1\subset F$ such that for all $(x_0,y_0)\in E_1\times F_1$, $T$ and $T^{-1}$
are differentiable at $x_0$ and $y_0$, respectively, in the sense that there exist $A, B$ symmetric such that for a.e. $(x,y)\in E\times F$,
 \begin{equation}\label{eq:difT}
  T(x)=T(x_0)+ A(x-x_0)+o(|x-x_0|) \qquad \textrm{and } \qquad T^{-1}(y)=T^{-1}(y_0)+ B(y-y_0)+o(|y-y_0|).
 \end{equation}
Moreover, we may assume that \eqref{eq:inverseT}
holds for every $(x_0,y_0)\in E_1\times F_1$. 
Using \eqref{eq:inverseT}, it is not hard to show that if $T(x_0)=y_0$, then $A=B^{-1}$ and then by \eqref{eq:Jacob} 
\begin{equation}\label{eq:Jacobfinal}
\rho_1(y_0) \det A=\rho_0(x_0).
\end{equation}
We finally let $E':= E_1\cap T^{-1}(F_1)$ and $F':=T(E')=F_1\cap T(E_1)$. 
Notice that since $T$ sends sets of measure zero to sets of measure zero, $|E\backslash E'|=|F\backslash F'|=0$. 
We are going to prove that $E'$ and $F'$ are open sets and that $T$ is a $C^{1,\alpha}$ diffeomorphism from $E'$ to $F'$.\\
Let $x_0\in E'$, and thus automatically $y_0:=T(x_0)\in F'$, be given; we shall prove that $T$ is of class $C^{1,\alpha}$  in a neighborhood of $x_0$.
%We notice that for any $R$ small enough, $B_R(x_0)\subset E$ and $B_R(y_0)\subset  F$. 
By \eqref{eq:difT} and the fact that $\rho_0$ and $\rho_1$ are bounded  we have in particular
\begin{equation}\label{hyp:epsBRx0}
\lim_{R\to 0} \frac{1}{R^2} \dashint_{B_R(x_0)} |T-y_0-A(x-x_0)|^2 \rho_0=0.
\end{equation}
We make the change of variables $ x=A^{-1/2}\hat{x} +x_0 $, $ y= A^{1/2}\hat{y}+y_0$, which leads to  $\hat T(\hat x):=A^{-1/2}( T(x)-y_0)$,
and then define $\hat \rho_0(\hat x):=\rho_0(x_0)^{-1}\rho_0(x)$ and $\hat \rho_1(\hat y):=  \rho_0(x_0)^{-1} \det^{-2} A \rho_1(y)$.
Note  that
$\hat T$ is the optimal transportation map between $\hat \rho_0$ and $\hat \rho_1$ (indeed, if $T=\nabla \psi$ for a convex function $\psi$,
then $\hat{T}= \hat{\nabla} \hat{\psi}$, where $\hat \psi(\hat x)=\psi(x)- y_0\cdot \hat x$) and that by \eqref{eq:Jacobfinal}, $\hat \rho_0(0)=\hat \rho_1(0)=1$. Moreover, 
since $\rho_0$ and $\rho_1$ are bounded and bounded away from zero, $\hat \rho_0$ and $\hat \rho_1$ are $C^{0,\alpha}$ continuous with H\"older semi-norms controlled by the ones of $\rho_0$ and $\rho_1$, so that 
\[
 \lim_{R\to 0} R^\alpha\lt([\hat \rho_0]_{\alpha, B_R}+ [\hat \rho_1]_{\alpha,R}\rt)=0.
\]
 Finally, the change of variables is made such that (\ref{hyp:epsBRx0}) turns into
\begin{align*}
\lim_{R\to 0} \frac{1}{R^2} \dashint_{B_R} |\hat{T}-\hat{x}|^2 \hat \rho_0=0.
\end{align*}
 Hence, we may apply Proposition \ref{P3} to $\hat{T}$ to obtain that $\hat{T}$ is of class $C^{1,\alpha}$ in a neighborhood of zero. 
Similarly, we obtain that $\hat{T}^{-1}$ is $C^{1,\alpha}$ in a neighborhood of zero.
Going back to the original map, this means that $T$ is a $C^{1,\alpha}$ diffeomorphism of a neighborhood $U$ of $x_0$ on the neighborhood $T(U)$ of $T(x_0)$. 
In particular, $U\times T(U)\subset E'\times F'$ so that $E'$ and $F'$ are both open and thanks to \eqref{eq:inverseT}, $T$ is a global $C^{1,\alpha}$ diffeomorphism from $E'$ to $F'$. \end{proof}

\begin{remark}
 If $\psi$ is a  convex function function such that $\nabla \psi =T$, Theorem \ref{theo:final} shows that in $\psi\in C^{2,\alpha}(E')$ and it solves (in the classical sense) the Monge-Amp\`ere equation which is now a uniformly elliptic equation. 
 If the densities are more regular then by the Evans-Krylov Theorem (see \cite{CafCab}) and Schauder estimates  we may obtain higher regularity of $T$.  
\end{remark}

\section*{Acknowledgment}
Part of this research was funded by the program PGMO of the FMJH through the  project COCA. The hospitality of the IHES and of the MPI-MIS are gratefully acknowledged.
\bibliographystyle{amsplain}
\bibliography{OT}
 \end{document}